\documentclass[11pt,A4paper]{article}

\RequirePackage[OT1]{fontenc}
\RequirePackage{amsthm,amsmath}
\RequirePackage[numbers]{natbib}
\RequirePackage[colorlinks,citecolor=blue,urlcolor=blue]{hyperref}
\usepackage{authblk}
\usepackage{amsmath}
\usepackage{amssymb}
\usepackage{graphicx}

\usepackage[mathscr]{eucal}
\usepackage{enumerate}

\usepackage{cases}

\usepackage{diagbox}

\numberwithin{equation}{section}
\theoremstyle{plain}
\newtheorem{thm}{Theorem}[section]
\newtheorem{coro}{Corollary}[section]
\newtheorem{lem}{Lemma}[section]
\newtheorem{prop}{Proposition}[section]
\newtheorem{rem}{Remark}[section]
\newtheorem{assumption}{Assumption}[section]
\numberwithin{equation}{section}

\newcommand{\E}{\mathbb{E}}
\newcommand{\PP}{\mathbb{P}}
\newcommand{\R}{\mathbb{R}}

\newcommand{\floor}[1]{\lfloor #1 \rfloor}

\begin{document}

\title{\bf Estimation of a pure-jump stable Cox-Ingersoll-Ross process}


\author[1]{\textsc{Elise Bayraktar}}
\author[1]{\textsc{Emmanuelle Cl\'ement}}
\affil[1]{LAMA, Univ Gustave Eiffel, Univ Paris Est Creteil, CNRS, F-77447 Marne-la-Vall\'ee, France.}


\date{Revision February, 08 2024}
\maketitle
\noindent
{\bf Abstract.}
We consider a pure-jump stable Cox-Ingersoll-Ross ($\alpha$-stable CIR) process driven by a non-symmetric stable L\'evy process with jump activity $\alpha \in (1,2)$ and we address the joint estimation of drift, scaling and jump activity parameters  from high-frequency observations of the process on a fixed time period. 
We first prove the existence of a consistent, rate optimal and asymptotically conditionally Gaussian estimator based on an approximation of the likelihood function. Moreover, uniqueness of the drift estimators is  established assuming that the scaling coefficient and the jump activity are known or consistently estimated.
Next we propose easy-to-implement preliminary estimators of all parameters and  we improve them by a one-step procedure.

\noindent
\textbf{MSC $2020$}.  60G51; 60G52; 60J75; 62F12.

\noindent
\textbf{Key words}: Cox-Ingersoll-Ross process, L\'evy process; parametric inference; stable process; stochastic differential equation.

\section{Introduction}

Stochastic equations with jumps are widely used to model time-varying phenomena in many fields such as finance and neurosciences.
Recently some extensions, including jumps, of the classical Cox-Ingersoll-Ross  process (CIR process) introduced in \cite{CIR}  to model interest rates evolution in finance have been considered in the literature
(see Jiao et al. \cite{BranchingCIR} and \cite{alpha-Heston}). These extensions, called $\alpha$-stable CIR processes, belong to the larger class of continuous-states branching processes with immigration (CBI processes) appearing in the description of large populations evolutions. An important literature is devoted to CBI processes, we refer to the recent survey paper by Li \cite{Li}, and their dynamics can
 be described by the stochastic equation (see Fu and Li \cite{ExistsSolution})
$$
X_t=x_0+at-b\int_0^t X_s ds + \sigma \int_0^t \sqrt{X_s} d B_s + \delta \int_0^t X_{s-}^{1/ \alpha} d L_s^{\alpha} \quad t \geq 0,
$$
where $(B_t)_{t \geq 0}$ is a standard Brownian motion and $(L_t^{\alpha})_{t \geq 0}$ a non-symmetric $\alpha$-stable L\'evy process with $\alpha \in (1,2)$.
In this paper, our aim is to study parametric estimation of the process from high-frequency observations on the time period $[0,T]$ with $T$ fixed and without loss of generality, we will assume that $T=1$. In this framework, we know that the drift parameters cannot be identified in presence of a Brownian Motion. Consequently, we focus in this work on the pure-jump stable CIR process ($\sigma=0$ in the previous equation), for which the estimation of all parameters from the fixed observation window $[0,1]$ is possible, and we consider the joint estimation of  the drift parameters $a$ and $b$, the scaling parameter $\delta$ and the jump activity $\alpha$.

 Before presenting our methodology, we briefly review existing estimation results  for the stable CIR process. Although there is an abundant literature devoted to the estimation of  stochastic equations with jumps, only few results relate to this process. Barczy et al. \cite{TransformeeLaplace}
studied the estimation of the parameter $b$ in presence of a Brownian component, assuming that the process is observed continuously on $[0,T]$ with $T$ going to infinity. Their approach is based on the explicit form of the likelihood function obtained from Girsanov's Theorem. Li and Ma \cite{LongNeg} considered the estimation of $a$ and $b$ by a regression method for a pure-jump stable CIR process observed discretely with fixed time-step and observation window growing to infinity. Yang \cite{Yang} estimated the parameters $a$, $b$ and $\delta$ from high-frequency observations of a pure-jump stable CIR process with small noise. Thus  joint estimation of all parameters (including the jump activity) has never been studied before for this process  and this is the purpose of this work. 

To perform the estimation of $(a,b,\delta,\alpha)$, the main idea is to approximate the conditional distribution of $X_{t+h}$ given $X_t$ by the stable distribution appropriately centered and rescaled. This methodology has been used in Masuda \cite{Masuda19} and also in Cl\'ement and Gloter \cite{ClementEstimating} for stochastic differential equations with Lipschitz coefficients, driven by symmetric stable processes. However these results do not apply to the pure-jump stable CIR process  as the driving L\'evy process is spectrally positive and the scaling coefficient $x^{1/ \alpha}$ is non-Lipschitz and non-bounded away from zero. Thus parametric estimation of this process requires a specific study. A preliminary key result is a bound of the $q$-moment of $1/X_t$, for $q>0$. We also establish some accurate rates of convergence in $L^p$-norm (with $p<\alpha$) for increments of the process in small time. These tools allow to control the error that arrises in approximating the conditional distribution of $X_{t+h}$ given $X_t$ by the stable distribution for small time-step $h$.

The paper is organized as follows. Section \ref{S:prelim} presents the pure-jump stable CIR process and gives some key moment estimates.  
Section \ref{S:EF} describes the estimation method, based on estimating equations. It also contains some limit theorems that allow to prove our main statistical results. The main results are stated in Section \ref{S:main}. We establish the existence of a rate optimal joint estimator, global uniqueness is also obtained for the estimation of the drift parameters. We also propose some first step estimators. Using power variations methods, we first estimate the jump activity $\alpha$ and the scaling parameter $\delta$. Next  plugging these estimators in the estimating equations, we estimate the drift parameters. Finally  the asymptotic properties of these first step estimators are improved by a one-step procedure.
The main results are proved in Section \ref{S:proofs}. 

Throughout the paper, we denote by $C$ (or $C_p$ if it depends on some parameter $p$) all positive  constants.

\section{Preliminaries} \label{S:prelim}

We consider a stable L\'evy process $(L_t^{\alpha})_{t \geq 0}$ with exponent $\alpha \in (1,2)$ defined on a filtered probability space  $(\Omega, \mathcal{F}, (\mathcal{F}_t)_{t \geq 0}, \PP)$, where  $(\mathcal{F}_t)_{t \geq 0}$ is the natural augmented filtration generated by $(L_t^{\alpha})_{t \geq 0}$.  We assume that  $(L_t^{\alpha})_{t \geq 0}$ is a  spectrally positive L\'evy process with characteristic function
\begin{equation}
\Phi(u)=\E(e^{iu L_1^{\alpha} })=e^{-|u|^{\alpha}(1- i \tan ( \pi \alpha/2) \text{sgn}(u))}. \label{E:char}
\end{equation}
This non-symmetric L\'evy process is strictly stable  (see for example Sato \cite{Sato}) and admits the representation 
$$
L^\alpha_t = \int_0^t \int_{\mathbb{R}_+} z d\tilde{N}(ds, dz),  \quad \tilde{N} = N - \overline{N},
$$
where  $N$ is a Poisson random measure with compensator $\overline{N}$ given by \[ \overline{N}(dt, dz) = dt F^{\alpha}(dz) \quad \text{for} \quad F^{\alpha}(dz) = \frac{C_\alpha}{z^{\alpha+1}} {\bf 1}_{]0, + \infty[} (z) dz.\] 
 From Lemma 14.11 in Sato \cite{Sato}, we have
$$
C_{\alpha}= -\frac{\alpha(\alpha-1)}{\Gamma(2-\alpha) \cos(\pi \alpha/2)}, \mbox{ with } \Gamma(\alpha)= \int_0^{\infty} x^{\alpha-1} e^{-x} dx, \quad \alpha>0.
$$
We will study parametric estimation of a pure-jump  stable Cox-Ingersoll-Ross process given by
 the stochastic differential equation 
\begin{equation} \label{eq:EDS}
    dX_t=(a-b X_t)dt + \delta X_{t-}^{1/ \alpha}dL^\alpha_t, \quad X_0=x_0  \quad t \in [0,1],
\end{equation}
where $a$, $b$, $\delta$ are real parameters. Assuming $a \geq 0$, $\delta \geq 0$ and $x_0 \geq 0$, 
 equation \eqref{eq:EDS} admits a pathwise unique strong solution (see Fu and Li \cite{ExistsSolution},  Li and Mytnik \cite{Li-Mytnik}) which is non-negative. It is also shown that $(X_t)_{t\in [0,1]}$ is a Markov process. 
  
 In all the paper we will assume the stronger condition $a>0$, $b \in \R$, $\delta >0$ and $x_0 >0$. Under this assumption, it is  established that $(X_t)_{t\in [0,1]}$ is positive (see  Foucart and Uribe Bravo \cite{Foucart} and Jiao, Ma and Scotti \cite{BranchingCIR}  ), that is $\PP( \forall t \in [0,1] ; X_t >0)=1$.
 
 Since $(X_t)$ is not bounded away from zero, we need to control $1/X_t$ and we first establish that $(X_t)$ admits moments of order $p < \alpha$. As mentioned in the introduction, the process $(X_t)_{t \geq 0}$ belongs to the class of  CBI processes  introduced by Kawasu and Watanabe \cite{KW} (we also refer to Li \cite{Li}) and consequently its Laplace transform has an explicit expression in terms of its branching mechanism and immigration rate. This allows to establish the following result,  where the case (ii) ($0<p<\alpha$) is a direct consequence of Proposition 2.8 in Li and Ma \cite{LongNeg}.

\begin{prop} \label{P:moments}

Let $(X_t)_{t\in [0,1]}$ be the solution of \eqref{eq:EDS}. Then we have 
\begin{equation} \label{eq:Moments1/X}
(i)  \quad \forall \; q > 0, \quad \sup_{t\in [0,1]} \E \left(\frac{1}{X_t^q} \right) < +\infty , 
\end{equation}
\begin{equation} \label{eq:MomentsX}
(ii) \quad  \forall \; 0 <p < \alpha, \; \forall s,t \in [0,1], \;  \E \left( \sup_{s \leq u \leq t} X_u^p  |\mathcal{F}_s\right) \leq C_p(1+ X_s^p). 
\end{equation}
\end{prop}
We now assume that we observe the discrete time Markov chain $(X_{\frac{i}{n}})_{0 \leq i \leq n}$, where $(X_t)_{t \in [0,1]}$  solves \eqref{eq:EDS}.
The next result gives some rate of convergence occuring in the discretization of the process $(X_t)$.  We mention that  (i) follows from Proposition 2.7 in Li and Ma \cite{LongNeg}.
\begin{prop}\label{P:intstoch}  Let $(X_t)_{t\in [0,1]}$ be the solution of \eqref{eq:EDS}. We have $ \forall \  0 < p < \alpha$
$$
\begin{array}{l}
(i)  \quad
\E \left( \sup_{s\in [\frac{i-1}{n}, \frac{i}{n}]} \left|  \int_\frac{i-1}{n}^s X_{t-}^{1/ \alpha} dL^\alpha_t  \right|^{p}|\mathcal{F}_\frac{i-1}{n} \right) \leq \frac{C_p}{n^{p/\alpha}}  \left( 1 + X_{\frac{i-1}{n}}^{p / \alpha} \right), \\
 (ii) \quad
        \E \left( \sup_{s\in [\frac{i-1}{n}, \frac{i}{n}]} \left|X_s - X_{\frac{i-1}{n}}\right|^{p}|\mathcal{F}_\frac{i-1}{n}\right) \leq \frac{C_p}{n^{p/\alpha}} \left( 1 + X_{\frac{i-1}{n}}^{p / \alpha} \right), 
\\
(iii) \quad
 \E \left( \left| \int_\frac{i-1}{n}^\frac{i}{n} (X_{s-}^{1/ \alpha} - X_{\frac{i-1}{n}}^{1/ \alpha}) dL^\alpha_s  \right|^{p} |\mathcal{F}_\frac{i-1}{n}\right) \leq \frac{C_p}{n^{2p/\alpha}} \left( 1 + \frac{1}{X_{\frac{i-1}{n}}^{p}} + X_{\frac{i-1}{n}}^p \right).
       \end{array}
       $$
\end{prop}


\section{Estimating functions and limit theorems} \label{S:EF}

In this paper, our aim is to estimate the parameters $(a,b,\delta, \alpha)$  and we will follow the methodology developed in \cite{ClementEstimating} based on 
estimating equations. From the previous estimates, we prove  that the rescaled increment 
$n^{1/\alpha} (X_{\frac{i}{n}}-X_{\frac{i-1}{n}} -\frac{a}{n} +\frac{b}{n} X_{\frac{i-1}{n}} )/ (\delta X_{\frac{i-1}{n}}^{1/ \alpha})$ converges in $L^p$-norm (for $1 \leq p < \alpha$) to $n^{1/\alpha}(L^{\alpha}_{\frac{i}{n}}-L^{\alpha}_{\frac{i-1}{n}})$. Moreover, using the scaling property of the L\'evy process $(L_t^{\alpha})_{t \geq 0}$, the distribution of  $n^{1/\alpha}(L^{\alpha}_{\frac{i}{n}}-L^{\alpha}_{\frac{i-1}{n}})$ is equal to the distribution of $L_1^{\alpha}$.  Consequently the conditional distribution of $X_{\frac{i}{n}}$ given $X_{\frac{i-1}{n}}=x$ can be approximated by the distribution of $x+a/n -bx/n +\delta x^{1/ \alpha} L_1^{\alpha} /n^{1/\alpha}$. Denoting by $\varphi_{\alpha}$ the density of the non-symmetric stable variable $L_1^{\alpha}$ with characteristic function \eqref{E:char}, we will then approximate the conditional distribution of $X_{\frac{i}{n}}$ given $X_{\frac{i-1}{n}}=x$ by
$$
\frac{n^{1/\alpha}}{\delta x^{1/ \alpha}} \varphi_{\alpha} \left(\frac{n^{1/\alpha}}{\delta x^{1/ \alpha}} ( y-x-\frac{a}{n}+\frac{bx}{n})\right). 
$$
 To estimate $\theta$, we therefore consider the following approximation of the log-likelihood function
\begin{equation}\label{def Ln}
     L_n (\theta) = \sum_{i=1}^n \log \left(\frac{n^{1/ \alpha}}{\delta X_{\frac{i-1}{n}}^{1/\alpha}}  \varphi_\alpha \left( n^{1/\alpha} \frac{X_{\frac{i}{n}} - X_{\frac{i-1}{n}} - \frac{a}{n} + \frac{b}{n} X_{\frac{i-1}{n}}}{\delta X_{\frac{i-1}{n}}^{1/ \alpha}} 
    \right) \right).
\end{equation}

The stable density $\varphi_{\alpha}$ has an explicit serie representation (see \cite{Sato} and \cite{Zolotarev}) but in practice, it is more convenient to use  the representation given in Nolan \cite{Nolan97}. To derive our results, we  do not need the explicit representation of $\varphi_{\alpha}$ but only the properties that we recall in \eqref{E : equivalents}.  

From now on we assume that  the observations $(X_{i/n})_{0 \leq i \leq n}$ are given by the stochastic equation \eqref{eq:EDS} for the parameter value $\theta_0=(\alpha_0, \beta_0, \delta_0, \alpha_0) \in (0, +\infty) \times \R \times (0, + \infty) \times (1,2) $. 

To simplify the presentation,
we write $L_t = L^{\alpha_0}_t$,  we introduce the notation $\Delta_i^n L = L_{\frac{i}{n}} - L_{\frac{i-1}{n}}$ and we set for $\theta=(a,b,\delta, \alpha)$
\begin{equation}\label{eq:def zn}
 z^n_i(\theta) = n^{1/\alpha} \left( \frac{X_{\frac{i}{n}} - X_{\frac{i-1}{n}} - \frac{a}{n} + \frac{b}{n} X_{\frac{i-1}{n}}}{\delta X_{\frac{i-1}{n}}^{1/ \alpha}} \right). 
 \end{equation} 

A natural choice of estimating function is the  approximation of the score function based on the quasi-likelihood function \eqref{def Ln}
\begin{equation} \label{def Gn} 
 G_n(\theta)= - \nabla_\theta  L_n (\theta) = - (\partial_a L_n (\theta),\partial_b L_n (\theta),\partial_{\delta} L_n (\theta),\partial_{\alpha} L_n(\theta))^T,
\end{equation}
where $M^T$ denotes the transpose of a matrix $M$.  We will estimate $\theta$ by solving the estimating equation $G_n(\theta) = 0$.
 
 Since the estimation method requires the computation of the score function $G_n$ and also of the hessian matrix $\nabla_{\theta} G_n$, we recall some properties of the density $\varphi_{\alpha}$ and introduce some more notation.
 
  The map $(x, \alpha) \mapsto \varphi_{\alpha}(x) $, defined on $\R \times (1,2)$, admits derivatives of any order with respect to $(x, \alpha)$  and from Theorem 2.5.1,  Corollary 2 and Theorem 2.5.2 in Zolotarev \cite{Zolotarev} (we also refer to Chap.3 -14 in Sato \cite{Sato}), we have the following equivalents.
For each $k>0$, $p>0$ there exist constants $C_{k,p,\alpha}$ and $C'_{k,p,\alpha}$ such that
\begin{equation}  \label{E : equivalents}
\partial^k_x \partial^p_\alpha \varphi_\alpha(x) \sim_{x \to +\infty} C_{k,p, \alpha} x^{-1-\alpha-k} (\ln x)^p,
\end{equation}
\[\partial^k_x \partial^p_\alpha \varphi_\alpha(x) \sim_{x \to -\infty} C'_{k,p, \alpha} \xi^\frac{2-\alpha+2k}{2\alpha}  (\ln |x|)^p e^{-\xi}, \quad \text{ where } \xi = (\alpha-1)\left| \frac{ x}{\alpha} \right|^\frac{\alpha}{\alpha -1}.\]
For $k=p=0$, $C_{0,0,\alpha}$ and $C'_{0,0,\alpha}$ are uniformly bounded away from 0 for $\alpha$ in any compact subset of $(1,2)$.
For $k\neq0$ or $p \neq 0$, $C_{k,p,\alpha}$ and $C'_{k,p,\alpha}$ are uniformly bounded for $\alpha$ in any compact subset of $(1,2)$.

From these equivalents, we deduce that $\E | L_1^{\alpha}|^p < \infty$ for $p \in (0,\alpha)$ and we also have
$$
\E(|L_1^{\alpha}|^q {\bf 1}_{L_1^{\alpha}<0}) <\infty, \quad \forall q >0.
$$
In what follows, we will use the  notation (where $f'=\partial_x f$): 
\begin{equation} \label{E:defh}
h_\alpha(x)=\frac{ \varphi'_\alpha}{\varphi_\alpha} (x),  \quad k_\alpha(x) = 1 + xh_\alpha(x), \quad f_\alpha(x) = \frac{\partial_\alpha \varphi_\alpha}{\varphi_\alpha} (x).
\end{equation}
We have $h_\alpha, k_\alpha, f_\alpha \in \mathcal{C}^\infty (\R \times (1,2), \R)$, note also that $ f'_\alpha = \partial_\alpha h_\alpha$. 

These  functions and their derivatives are not bounded (contrarily to the symmetric stable case) but
from the previous equivalents they satisfy, as well as their derivatives, Assumption \ref{A:conditions h} below.

\begin{assumption} \label{A:conditions h}
$h: \R \times (1,2) \to \R$ is continuously differentiable and for any compact set $A \subset (1,2)$, there exist  $C >0$ and $q>0$ such that $  \forall x \in \R$
\begin{equation*}
   \sup_{\alpha \in A} \left( |h(x, \alpha)| + |\partial_x h (x, \alpha)| + |\partial_\alpha h(x,\alpha)| \right) \leq C \left( 1 + |x|^q {\bf 1}_{x<0} + (\ln(1+x))^q {\bf 1}_{x>0} \right).
  \end{equation*} 
\end{assumption}
This allows to deduce that $h_\alpha, k_\alpha, f_\alpha$ and their derivatives are integrable with respect to the density $\varphi_{\alpha}$. Moreover 
we can  prove that
 \begin{equation} \label{esp func = 0}
 \E (h_{\alpha} (L^{\alpha}_1))= 0, \quad  \E (k_{\alpha} (L^{\alpha}_1) ) =0, \quad  \E (f_{\alpha} (L_1^\alpha))=0,
\end{equation}
 where the first equality follows from $\int \varphi'_{\alpha}(x) dx=0$, the second one from the integration by part formula $\int x \varphi'_{\alpha}(x)dx=-\int\varphi_{\alpha}(x) dx=-1$ and the third one from Lebesgue dominated Theorem $\int \partial_{\alpha} \varphi_{\alpha}(x) dx=0$.
 With similar arguments, we also have the connections
\begin{align} 
    \E( h'_{\alpha}(L^{\alpha}_1)) & = - \E( h_{\alpha}^2(L^{\alpha}_1)),  
    & \E( f'_{\alpha}(L^{\alpha}_1)) & = - \E((f_{\alpha}h_{\alpha})(L^{\alpha}_1)), \nonumber \\
    \E((\partial_{\alpha} f_{\alpha})(L^{\alpha}_1)) & = - \E( f_{\alpha}^2(L^{\alpha}_1)),  
    & \E( k'_{\alpha}(L^{\alpha}_1)) & = - \E( ( h_{\alpha}k_{\alpha})(L^{\alpha}_1)), \label{esp func} \\
    \E(L_1^{\alpha} f'_{\alpha}(L^{\alpha}_1)) & = - \E((f_{\alpha}k_{\alpha})(L^{\alpha}_1)),
    & \E(L_1^{\alpha} k'_{\alpha}(L^{\alpha}_1) ) & = - \E( k^2_{\alpha}(L_1)). \nonumber
\end{align}
We just detail the first two ones. Since $\int h'_{\alpha}(x) \varphi_{\alpha}(x) dx= - \int h_{\alpha}(x)\varphi'_{\alpha}(x)dx= -\int h_{\alpha}(x)^2 \varphi_{\alpha}(x) dx$, we deduce the first result. Combining that $\partial_{\alpha} f_{\alpha}=\partial^2_{\alpha} \varphi_{\alpha}/\varphi_{\alpha}- f_{\alpha}^2$ with $\int \partial^2_{\alpha} \varphi_{\alpha}(x)dx=0$, we obtain $ \E((\partial_{\alpha} f_{\alpha})(L^{\alpha}_1))  = - \E( f_{\alpha}^2(L^{\alpha}_1))$.

Using the previous notation, we now give an explicit expression of the estimating function $G_n$.
We first compute
 the partial derivatives of $z^n_i(\theta)$
\[\partial_az^n_i(\theta) = - \frac{n^{1/ \alpha}}{n} \frac{1}{ \delta X_{\frac{i-1}{n}}^{1/ \alpha}}, 
\quad \partial_bz^n_i(\theta) = \frac{n^{1/ \alpha}}{n} \frac{X_{\frac{i-1}{n}}}{ \delta X_{\frac{i-1}{n}}^{1/ \alpha}}, 
\quad \partial_\delta z^n_i(\theta) = -\frac{z^n_i(\theta)}{\delta},\]
\[ \partial_\alpha z^n_i(\theta) =- \frac{ \ln \left( n / X_{\frac{i-1}{n}} \right)}{\alpha^2}z^n_i(\theta),\]
and we obtain
\begin{equation}\label{E:Gn1}
    G_n^1(\theta)= -\partial_a L_n(\theta) = \frac{n^{1/ \alpha}}{n} \sum_{i=1}^n  \frac{1}{\delta X_{\frac{i-1}{n}}^{1/ \alpha}} h_\alpha(z^n_i(\theta)),
\end{equation}
\begin{equation} \label{E:Gn2}
    G_n^2(\theta)= -\partial_b L_n(\theta) =  - \frac{n^{1/ \alpha}}{n}\sum_{i=1}^n  \frac{X_{\frac{i-1}{n}}}{\delta X_{\frac{i-1}{n}}^{1/ \alpha}} h_\alpha(z^n_i(\theta)),
\end{equation}
\begin{equation} \label{E:Gn3}
   G_n^3(\theta)= -\partial_{\delta} L_n(\theta) =  \sum_{i=1}^n \frac{1}{\delta} k_\alpha(z^n_i(\theta)),
\end{equation}
\begin{equation} \label{E:Gn4}
    G_n^4(\theta)= -\partial_{\alpha} L_n(\theta) =  \sum_{i=1}^n \left( \frac{\ln \left( n / X_{\frac{i-1}{n}} \right)}{\alpha^2} k_\alpha(z^n_i(\theta)) - f_\alpha(z^n_i(\theta)) \right).
\end{equation}
To state our main statistical results, 
 we need to establish some limit theorems for such functions, after appropriate normalizations.
 The first result is an uniform law of large numbers, where the uniformity is obtained for $(a,b)$ in any compact subset $A$ of $(0, +\infty) \times \R$ and for $(\delta, \alpha) \in W_n^{(\eta)}$ a neighborhood of $(\delta_0, \alpha_0)$  defined for $\eta>0$ by
\begin{equation} \label{Voisinage Wn}
        W_n^{(\eta)} = \{ (\delta, \alpha) ; \left| \left| w_n^{-1}\begin{pmatrix} \delta - \delta_0 \\
    \alpha - \alpha_0 \end{pmatrix} \right| \right| \leq \eta\} \quad \text{ with } (\ln n)^q w_n \rightarrow 0\; \forall q>0 ,
 \end{equation}
 where $|| \cdot ||$ denotes the Euclidean norm.
This uniform law of large numbers is stated with a rate of convergence $(\ln n)^q$ for $q>0$ to compensate for the loss of a factor $\ln(n)$ in the joint estimation of the parameters.
\begin{thm} \label{Th : LFGN}
Let $h$  satisfying  Assumption \ref{A:conditions h}  and let $f:(0, + \infty) \times (0, +\infty) \times (1,2) \to \R $ be continuously differentiable. We assume that for all compact set $\overline{A} \subset (0, + \infty) \times (1,2)$, there exist   $C >0$  and $q>0$ such that $ \forall x\in (0, + \infty)$
\begin{equation} \label{eq : conditions f}
 \sup_{(\delta, \alpha) \in \overline{A}} \left(|f (x, \delta, \alpha)| + |f' (x, \delta, \alpha)| +|\partial_{(\delta, \alpha)} f (x, \delta, \alpha)| \right) \leq C \left( 1 + x^q  + \frac{1}{x^q} \right).
\end{equation} 
 With $W_n^{(\eta)}$ defined in \eqref{Voisinage Wn}, for any compact set $A \subset (0,  +\infty) \times \R$  and $\forall q>0$ we have  the convergence in probability to zero of
    \begin{align*}
    \sup_{\theta \in A \times W_n^{(\eta)}  }  (\ln n)^q\left| \frac{1}{n}\sum_{i=1}^n f(X_{\frac{i-1}{n}}, \delta, \alpha)h (z_i^n(\theta),\alpha) - \int_0^1 f(X_s, \delta_0, \alpha_0)ds \E h(L_1,\alpha_0) \right|.  
\end{align*}
\end{thm}
We next state a Central Limit Theorem for triangular arrays. This CLT is obtained for the stable convergence in law which is stronger than the usual convergence in law and appropriate to our statistical problem.  This convergence is denoted by
$\xrightarrow[]{\mathcal{L}-s}$ and we refer to Jacod and Protter \cite{TriangularArray} for the definition and properties of the stable convergence in law.   

\begin{thm} \label{Th : TCL}
Let $H = (h_j)_{1 \leq j \leq d} : \R \to \R^d$. We assume that $\forall \; j \in \{1, ... , d \}$
$h_j$ satisfies Assumption \ref{A:conditions h} and $\E(h_j(L_1)) = 0$.
Let $F = (f_{kj})_{1 \leq k,j \leq d} : (0, + \infty) \to \R^{d \times d} $ be a continuous function. We assume that there exist  $C >0$ and $q>0$ such that
\begin{equation}\label{eq : conditions f i}
    \forall x\in (0, + \infty), \quad ||F (x)|| \leq C \left( 1 + x^q  + \frac{1}{x^q} \right).
\end{equation} 
Then we have the stable convergence in law with respect to $\sigma(L_s, s \leq 1)$
\begin{equation}
    \frac{1}{\sqrt{n}} \sum_{i=1}^n F(X_{\frac{i-1}{n}}) H(z^n_i(\theta_0))  \xrightarrow[n \to \infty]{\mathcal{L}-s} \Sigma^{1/2} \mathcal{N},
\end{equation}
where $\mathcal{N}$ is a standard Gaussian variable independent of $\Sigma$ and 
\begin{equation*}
    \Sigma = \int_0^1 F(X_s)  \Sigma' F(X_s)^T ds \;
\text{ with } \;    \Sigma'=(\Sigma'_{k,j}), \; \Sigma'_{k,j}=\E (h_k h_j(L_1)) \; 1 \leq k,j  \leq d .
\end{equation*}
\end{thm}
With this background, we turn to the estimation problem.

\section{Main results} \label{S:main}
\subsection{Joint estimation}
We first prove the existence of a rate optimal joint estimator. To this end
we consider $u_n$ the block-diagonal  rate matrix
\begin{equation}
    u_n = \begin{pmatrix} \label{eq:un}
\frac{1}{n^{1/ \alpha_0 - 1/2}} Id_2 & 0\\
 0 & \frac{1}{\sqrt{n}} v_n
\end{pmatrix} \quad \mbox{ where } 
 Id_2=\begin{pmatrix}
1 & 0 \\
0 & 1 \\
\end{pmatrix},
\quad  v_n = \begin{pmatrix}
v_n^{11} & v_n^{12} \\
v_n^{21} & v_n^{22} \\
\end{pmatrix}.
\end{equation}
We assume that 
\begin{equation} \label{E:def-vr}
r_n(\theta_0) v_n \to \overline{v} = \begin{pmatrix}
\overline{v}_{11} & \overline{v}_{12} \\
\overline{v}_{21} & \overline{v}_{22} \\
\end{pmatrix}
\; 
\text{with} \; \det(\overline{v}) >0 \; \text{and} \;
 r_n(\theta) = \begin{pmatrix}
    \frac{1}{\delta} & \frac{\ln(n)}{\alpha^2} \\
    0 & 1
\end{pmatrix}. 
\end{equation}
We can choose for example $v_n = r_n^{-1}(\theta_0)$ and in that case $\overline{v} = Id_2$.

We next define the information matrix $I(\theta)$ by
\begin{equation} \label{E:Info}
I(\theta) = \begin{pmatrix}
    I^{11} (\theta) &  I^{21} (\theta)^T\\
     I^{21} (\theta) &  I^{22} (\theta) 
\end{pmatrix}. 
\end{equation}
$I^{11} (\theta)$ and $I^{22} (\theta)$ are symmetric matrices  (for such symmetric matrices, we only write the lower diagonal terms and just write $symm$ for the upper diagonal terms) given respectively by
\[I^{11} (\theta) = \begin{pmatrix}
    \frac{1}{\delta^2} \int_0^1 \frac{1}{X_s^{2/ \alpha}} ds \E( h^2_{\alpha} (L^{\alpha}_1)) 
    & \text{symm}  \\
   - \frac{1}{\delta^2} \int_0^1 \frac{X_s}{X_s^{2/ \alpha}} ds \E( h^2_{\alpha} (L^{\alpha}_1)) 
    & \frac{1}{\delta^2} \int_0^1 \frac{X_s^2}{X_s^{2/ \alpha}} ds \E( h^2_{\alpha} (L^{\alpha}_1)) 
\end{pmatrix}, \]
 \[I^{22} (\theta) = \begin{pmatrix}
    \E( k^2_{\alpha} (L^{\alpha}_1)) 
    &  \text{symm} \\
  -  \frac{1}{\alpha^2} \int_0^1 \ln(X_s) ds \E( k^2_{\alpha} (L^{\alpha}_1)) - \E(f_{\alpha}k_{\alpha} (L^{\alpha}_1)) 
    &  I_{22}^{22} (\theta)
\end{pmatrix} ,\]
with
\[
I_{22}^{22}(\theta) =\frac{1}{\alpha^4} \int_0^1 (\ln X_s)^2 ds \E( k^2_{\alpha} (L^{\alpha}_1))
    + \frac{2}{\alpha^2} \int_0^1 \ln(X_s) ds \E( f_{\alpha}k_{\alpha} (L^{\alpha}_1)) 
    + \E( f_{\alpha}^2 (L^{\alpha}_1)). 
\]
The matrix $I^{21} (\theta)$ is given by
\[I^{21} (\theta) = \begin{pmatrix}
    \frac{1}{\delta} \int_0^1 \frac{ds}{X_s^{1/ \alpha}} \E( h_{\alpha}k_{\alpha} (L^{\alpha}_1)) 
    & - \frac{1}{\delta} \int_0^1 \frac{X_s }{X_s^{1/ \alpha}}ds \E( h_{\alpha}k_{\alpha} (L^{\alpha}_1))  \\
  I^{21}_{21}(\theta)
    &  I^{21}_{22}(\theta) 
    \end{pmatrix}, \]
   where
\begin{align*}
I^{21}_{21}(\theta)= - \frac{1}{\delta \alpha^2} \int_0^1 \frac{\ln(X_s)}{X_s^{1/ \alpha}} ds \E( h_{\alpha}k_{\alpha} (L^{\alpha}_1)) - \frac{1}{\delta} \int_0^1 \frac{ds}{X_s^{1/ \alpha}} \E( f_{\alpha}h_{\alpha} (L^{\alpha}_1)), \\
I^{21}_{22}(\theta) = \frac{1}{\delta \alpha^2} \int_0^1 \frac{\ln(X_s) X_s}{X_s^{1/ \alpha}} ds \E( h_{\alpha}k_{\alpha} (L^{\alpha}_1)) + \frac{1}{\delta} \int_0^1 \frac{X_s }{X_s^{1/ \alpha}}ds \E( f_{\alpha}h_{\alpha} (L^{\alpha}_1)). 
\end{align*}
With this notation, we can state our main result.

\begin{thm}\label{Th-local}
    Let $G_n$ be defined by \eqref{def Gn}. Then there exists an estimator $(\hat{\theta}_n)$ solving the equation $G_n(\theta) = 0$ with probability tending to 1, that converges in probability to $\theta_0$. 
   The information matrix $I(\theta_0)$ defined in  \eqref{E:Info} is positive definite and
     we have the stable convergence in law with respect to $\sigma(L_s, s \leq 1)$  \begin{equation}
        u_n^{-1} (\hat{\theta}_n - \theta_0) \xrightarrow[n \to \infty]{\mathcal{L}-s} I_{\overline{v}}(\theta_0)^{-1/2} \mathcal{N},
\end{equation}
where $\mathcal{N}$ is a standard Gaussian variable independent of $I_{\overline{v}}(\theta_0)$ and  $I_{\overline{v}}(\theta_0)$ is defined by 
\[I_{\overline{v}}(\theta_0) = \begin{pmatrix}
    Id_2 & 0 \\
0    &\overline{v}^T 
\end{pmatrix}  I(\theta_0) 
\begin{pmatrix}
    Id_2 & 0 \\
0    &\overline{v} 
\end{pmatrix} \; \text{ with } \; \overline{v} \; \text{ given by } \;\eqref{E:def-vr}.
\]
\end{thm}
\begin{rem}
 Let us discuss the optimality of the previous estimation procedure. The  Local Asymptotic Normality (LAN) property or Local Asymptotic Mixed Normality (LAMN) property allows to identify the optimal rate of convergence and the minimal asymptotic variance in estimating a parameter $\theta$.
The LAN property  has been established  in Brouste and Masuda \cite{EfficientEstimatorMasuda} for the joint estimation of $(a, \delta, \alpha)$ based on high-frequency observations of the process  
$$
X_t = a t + \delta S_t^{\alpha},
$$ 
where $(S_t^{\alpha})$ is  a symmetric $\alpha$-stable L\'evy process. Extending their result to the non-symmetric L\'evy process $(L_t^{\alpha})$, one can show that the LAN property holds in the non-symmetric case with rate $u_n=\text{diag}(\frac{1}{n^{1/ \alpha - 1/2}}, \frac{1}{\sqrt{n}} r_n^{-1}(\theta))$ and information given by
$$
 I(a, \delta, \alpha)=\begin{pmatrix}
\frac{1}{\delta^2} \E h_{\alpha}^2(L_1^{\alpha}) & \frac{1}{\delta} \E h_{\alpha} k_{\alpha}(L_1^{\alpha}) & -\frac{1}{\delta} \E f_{\alpha} h_{\alpha}(L_1^{\alpha}) \\
 \frac{1}{\delta} \E h_{\alpha} k_{\alpha}(L_1^{\alpha})  & \E k^2_{\alpha} (L_1^{\alpha}) & - \E f_{\alpha} k_{\alpha}(L_1^{\alpha}) \\
 -\frac{1}{\delta} \E f_{\alpha} h_{\alpha}(L_1^{\alpha}) & - \E f_{\alpha} k_{\alpha}(L_1^{\alpha}) & \E f^2_{\alpha}(L_1^{\alpha})
 \end{pmatrix}.
$$
For the pure-jump CIR process, the optimality of estimation procedure is still an open problem (the LAMN property has not been yet established and only partial results exist for pure-jump equations, see for example \cite{CGN}) but we conjecture that our estimator is  rate optimal and probably efficient.
\end{rem}
The previous result states the existence of a consistent and rate optimal estimator but does not ensure the existence of a unique root to the equation $G_n(\theta)=0$, there might be other sequences solution to the estimating equation that are not consistent. However, global uniqueness can be obtained in the drift estimation if we assume that $\delta_0$ and $\alpha_0$ are known or are consistently estimated by $\tilde{\delta}_n$ and $\tilde{\alpha}_n$. We focus now on the estimation of the drift parameters $(a,b)$ and consider the estimating function restricted to the drift
\begin{equation} \label{E:Gndrift}
G_n^{(d)}(a,b) = - \nabla_{(a,b)}  L_n (a,b,\tilde{\delta}_n, \tilde{\alpha}_n) = - (\partial_a L_n (a,b,\tilde{\delta}_n, \tilde{\alpha}_n),\partial_b L_n (a,b,\tilde{\delta}_n, \tilde{\alpha}_n))^T.
\end{equation}
In that case, we can prove global uniqueness.
\begin{thm}\label{Th-global}
We assume that $(a_0, b_0)$  belongs to the interior of $A$, a compact subset of $ (0,  + \infty) \times \R$.
    Let $G_n^{(d)}$ be defined by \eqref{E:Gndrift} 
  with $\frac{\sqrt{n}}{\ln n}(\tilde{\delta}_n-\delta_0, \tilde{\alpha}_n- \alpha_0)$  tight.
Then any sequence $(\hat{a}_n, \hat{b}_n)$ that solves $G_n^{(d)}(a,b)=0$  converges in probability to $(a_0,b_0)$. 
Such a sequence exists and is unique in the sense that if  $(\hat{a}'_n, \hat{b}'_n)$ is another sequence that solves $G_n^{(d)}(a,b)=0$ then $\PP((\hat{a}'_n, \hat{b}'_n) \neq (\hat{a}_n, \hat{b}_n))$ goes to zero as $n$ goes to infinity.  Moreover $\frac{n^{1/ \alpha_0 - 1/2}}{(\ln n)^{2}} \begin{pmatrix}
            \hat{a}_n - a_0\\
            \hat{b}_n - b_0
        \end{pmatrix}$ is tight.
\end{thm}
In the next section, we give preliminary estimators  $(\tilde{\delta}_n, \tilde{\alpha}_n)$ satisfying
the assumptions of Theorem \ref{Th-global}.
Obviously, if $\delta_0$ and $\alpha_0$ are known then $(\hat{a}_n, \hat{b}_n)$ has the asymptotic distribution of Theorem \ref{Th-local}
\begin{equation*}
        n^{1/ \alpha_0 - 1/2} \begin{pmatrix}
            \hat{a}_n - a_0\\
            \hat{b}_n - b_0
        \end{pmatrix} \xrightarrow[n \to \infty]{\mathcal{L}-s} I^{11}(\theta_0)^{-1/2} \mathcal{N} \; \text{with}  \; I^{11}(\theta_0) \; \text{defined in} \; \eqref{E:Info}.
\end{equation*}

\begin{rem}
   We have not been able to prove a global result in estimating $\theta$ due to the non-symmetry of the stable process.
    \end{rem}

\subsection{Preliminary estimators and one-step improvement} \label{S:onestep}

We now propose  some preliminary estimators of  the jump activity $\alpha_0$ and the scaling parameter $\delta_0$. We will use  well known power variation methods  studied for stochastic processes with jumps in Barndorff-Nielsen et al. \cite{IntroPowerVarSauts}. We also refer to Woerner  \cite{EstimationPowerVarWoerner}, Todorov and Tauchen \cite{EstimationPowerVariationTodorov}, Todorov  \cite{Todorov}.

Following Todorov \cite{Todorov}, we define the p-th order power-variation of the second order differences of a pure-jump semimartingale X by
\[V_n^1(p,X) = \sum_{i=2}^n |\Delta_i^n X - \Delta_{i-1}^n X|^p \quad \text{where } \Delta_i^n X = X_{\frac{i}{n}} - X_{\frac{i-1}{n}}, \]   
    \[V_n^2(p,X) = \sum_{i=4}^n |\Delta_i^n X - \Delta_{i-1}^n X + \Delta_{i-2}^n X - \Delta_{i-3}^n X|^p, \]
and consider the estimator of the jump activity  
 \begin{equation}\label{Estimateur alpha}
        \tilde{\alpha}_n(p) = \frac{p \log 2}{  \log (V_n^2(p, X) / V_n^1(p, X))} {\bf 1}_{V_n^1(p, X) \neq V_n^2(p, X)}.
 \end{equation}
From Corollary 1 in \cite{Todorov}, we know that  $\tilde{\alpha}_n(p)$ is consistent if $p < \alpha_0$ and if $\frac{\alpha_0-1}{2} <p < \frac{\alpha_0}{2}$,  $\sqrt{n}(\tilde{\alpha}_n(p) - \alpha_0)$ stably converges in law.
Since we have $\alpha_0 > 1$,  we can use this result with $p=1/2$ to estimate $\alpha_0$.

Next we estimate $\delta_0$ by plugging $\tilde{\alpha}_{n}(\frac{1}{2})$ in the normalized $1/2$-power variation of second order differences of the process $X$.  Let  $\overline{L}_1^{\alpha}$ be an independent  copy of  $L_1^{\alpha}$ and set
\begin{equation} \label{E:moment-stable}
m_p(\alpha) = \E |L^{\alpha}_1 - \overline{L}^\alpha_1 |^{p}.
\end{equation}
Using \eqref{E:char}, we check that $\E(e^{iu(L^\alpha_1 - \overline{L}^\alpha_1)}) = e^{-2|u|^{\alpha}}$. So $L^\alpha_1 - \overline{L}^\alpha_1$ has a symmetric stable distribution and from Masuda \cite{Masuda} we have for $0<p< \alpha$
\begin{equation} \label{E:mp}
m_p(\alpha)=\frac{2^{p/ \alpha} 2^p \Gamma( \frac{p+1}{2}) \Gamma(1- \frac{p}{\alpha})}{ \sqrt{\pi}\Gamma(1- \frac{p}{2})} \quad
\text{with} \; \Gamma(a)= \int_0^{\infty} x^{a-1} e^{-x} dx.
\end{equation}
We next define for $p< \alpha_0$
\begin{equation} \label{E:pre-delta-p}
   \tilde{\delta}_{n}( p) = \frac{1}{m_p(\tilde{\alpha}_{n}(p))} \frac{1}{n} \sum_{i=2}^n n^{p/  \tilde{\alpha}_{n}(p)} \frac{\left| \Delta_i^n X - \Delta_{i-1}^n X\right|^{p}}{X_{\frac{i-2}{n}}^{p/ \tilde{\alpha}_{n}(p)}}.
\end{equation}
We will prove that the estimator $\tilde{\delta}_n(1/2)$ converges in probability to  $\delta_0^{1/2}$ with rate of convergence  $ \frac{\sqrt{n}}{\ln(n)}$ and we estimate $\delta_0$ by    $[\tilde{\delta}_n(\frac{1}{2})]^2 $.
\begin{thm}\label{Estimation delta} 
Let $\tilde{\delta}_n= [\tilde{\delta}_n(\frac{1}{2})]^2$, with $\tilde{\delta}_n(p)$ defined by \eqref{E:pre-delta-p}. Then $\tilde{\delta}_n$ converges in probability to $\delta_0$ and $ \frac{\sqrt{n}}{\ln(n)} (\tilde{\delta}_n - \delta_0)$ is tight.
\end{thm}

From the preliminary estimators $\tilde{\delta}_n$ and $\tilde{\alpha}_n:=\tilde{\alpha}_n(1/2)$, and using the global uniqueness result of Theorem \ref{Th-global}, we estimate $(a,b)$ with the estimating equation $G_n^{(d)}$ defined by \eqref{E:Gndrift} and denote by $(\tilde{a}_n, \tilde{b}_n)$ the resulting estimators. We obtain a preliminary estimator $\hat{\theta}_{0,n} = (\tilde{a}_n, \tilde{b}_n, \tilde{\delta}_n, \tilde{\alpha}_n)$ which is consistent and from Theorems \ref{Th-global} and \ref{Estimation delta}, we know that $\ln(n)^{-2} u_n^{-1} ( \hat{\theta}_{0,n} - \theta_0)$ is tight.
We now describe a one-step method to improve the asymptotic properties of $\hat{\theta}_{0,n}$. We also refer to Masuda \cite{Masuda23} where a one-step method is proposed for the estimation of Ornstein-Ulhenbeck type processes. 

We define  $\hat{\theta}_{1,n}$  by
\begin{equation}\label{One step improvement estimator}
    \hat{\theta}_{1,n} = \hat{\theta}_{0,n} -J_n(\hat{\theta}_{0,n})^{-1} G_n(\hat{\theta}_{0,n}),
\end{equation}
where $J_n(\theta)= \nabla_{\theta}G_n(\theta)$. It is shown in Section \ref{Ss:Thlocal}  that the convergence in probability holds
\begin{equation} \label{E:unifJn}
\sup_{(a,b) \in A, \; (\delta, \alpha) \in W_n^{(\eta)}} || u_n^T J_n( \theta) u_n -I_{\overline{v}}(\theta_0) || \xrightarrow{} 0,
\end{equation}
where $I_{\overline{v}}(\theta_0)$  is  positive definite,  $W_n^{(\eta)}$ is defined by \eqref{Voisinage Wn} and $A$ is a compact subset of $(0, +\infty) \times \R$. We deduce then that $u_n^T J_n(\hat{\theta}_{0,n}) u_n$ converges to $I_{\overline{v}}(\theta_0)$. Let us denote by $D_n$ the set where   $u_n^TJ_n(\hat{\theta}_{0,n})u_n$ is invertible. We have   $\PP(D_n) \rightarrow 1$ and $\hat{\theta}_{1,n}$ is well defined on this set.
Considering $\hat{\theta}_{n}$ the estimator defined in Theorem \ref{Th-local}, we will prove that  $u_n^{-1} (\hat{\theta}_{1,n} -  \hat{\theta}_{n}) \xrightarrow{} 0$ in probability. 

From now on, we assume that $u_n$ is given by \eqref{eq:un} with $v_n= r_n^{-1}(\theta_0)$ (defined  in \eqref{E:def-vr}) such that $\overline{v}=Id$.
We have $G_n(\hat{\theta}_{n})=0$ on $\overline{D}_n$ with 
$\PP(\overline{D}_n) \rightarrow 1$ and so from Taylor's formula, we have on  $\overline{D}_n \cap D_n$
\[\hat{\theta}_{1,n} = \hat{\theta}_{0,n} - J_n(\hat{\theta}_{0,n})^{-1} \int_0^1 J_n(\hat{\theta}_{n} + t(\hat{\theta}_{0,n} - \hat{\theta}_{n}))dt \ (\hat{\theta}_{0,n} - \hat{\theta}_{n}). \]
Hence
\begin{align*}
  \hat{\theta}_{1,n} -  \hat{\theta}_{n}  =  J_n(\hat{\theta}_{0,n})^{-1} \left( J_n(\hat{\theta}_{0,n}) - \int_0^1 J_n(\hat{\theta}_{n} + t(\hat{\theta}_{0,n} - \hat{\theta}_{n}))dt \right) (\hat{\theta}_{0,n} - \hat{\theta}_{n}), 
\end{align*}
and we deduce
\begin{align*}
    u_n^{-1} (\hat{\theta}_{1,n} -  \hat{\theta}_{n})
    =  & (u_n^T J_n(\hat{\theta}_{0,n})u_n)^{-1} 
  \left[u_n^T J_n(\hat{\theta}_{0,n}) u_n  \right. \\
 & \left.  -  u_n^T \int_0^1 J_n(\hat{\theta}_{n} + t(\hat{\theta}_{0,n} - \hat{\theta}_{n}))dt \ u_n \right] 
    u_n^{-1}(\hat{\theta}_{0,n} - \hat{\theta}_{n}).
\end{align*}
Using the tightness of $ (\ln n)^{-2} u_n^{-1} (\hat{\theta}_{0,n} - \theta_0)$, we may assume that $\hat{\theta}_{0,n} \in A \times W_n^{(\eta)}$ (as well as $\hat{\theta}_{n} + t(\hat{\theta}_{0,n} - \hat{\theta}_{n})$ for $t \in [0,1]$), so we obtain from \eqref{E:unifJn} and Remark \ref{R:Jn} that $\forall q>0$
$$
(\ln n)^q \left\Vert u_n^T J_n(\hat{\theta}_{0,n}) u_n  -  u_n^T \int_0^1 J_n(\hat{\theta}_{n} + t(\hat{\theta}_{0,n} - \hat{\theta}_{n}))dt \ u_n \right\Vert \rightarrow 0,
$$
and it yields the convergence in probability $u_n^{-1} (\hat{\theta}_{1,n} -  \hat{\theta}_{n}) \xrightarrow{} 0$.
Recalling that $\hat{\theta}_{n}$  satisfies
$$
 u_n^{-1} (\hat{\theta}_{n} - \theta_0) \xrightarrow[n \to \infty]{\mathcal{L}-s} I_{\overline{v}}(\theta_0)^{-1/2} \mathcal{N},
$$
we obtain the following result.

\begin{coro}
    The one-step estimator $\hat{\theta}_{1,n}$ converges in probability to $\theta_0$ and we have the stable convergence in law
\begin{equation*}
    u_n^{-1} (\hat{\theta}_{1,n} -  \theta_0) \xrightarrow[n \to \infty]{\mathcal{L}-s} I(\theta_0)^{-1/2} \mathcal{N},
\end{equation*}
with $I(\theta_0)$ defined by \eqref{E:Info}, and $u_n$ given by \eqref{eq:un} with $v_n=r_n^{-1}(\theta_0)$.
\end{coro} 

We conclude with some comments on the implementation of our estimation method. From \eqref{Estimateur alpha}, \eqref{E:pre-delta-p} and \eqref{E:mp} with $p=1/2$, the estimators $\tilde{\alpha}_n$ and $\tilde{\delta}_n$ are very simple to compute. Moreover, from the proof Theorem \ref{Th-global} $(\tilde{a}_n, \tilde{b}_n)$ is the unique maximum of the  approximated log-likelihood function $L_n(a,b, \tilde{\delta}_n, \tilde{\alpha}_n)$ given by \eqref{def Ln}, where the density of the non-symmetric stable distribution $\varphi_{\alpha}$ can be evaluated  from the results of Nolan \cite{Nolan97}. Finally, the score function $G_n$ and the hessian matrix $J_n$ appearing in the one-step correction can be computed using
finite differences.

\section{Proofs} \label{S:proofs}

\subsection{Proof of Proposition \ref{P:moments}}

 Only (i) requires a proof since (ii) is stated in Proposition 2.8 of \cite{LongNeg}.

We recall that $\forall t \geq 0, X_t > 0$. By Fubini and integration by parts, we can show that 
\[\E\left(\frac{1}{X_t^q}\right) = C_q \int_{\mathbb{R}_+} u^{q-1} \E(e^{-u X_t}) du.\]
The Laplace transform  of $X_t$ has the explicit expression (we refer to Li \cite{Li}, see also Section 3.1 in \cite{BranchingCIR} and Proposition 2.1. in \cite{TransformeeLaplace})
\begin{equation} \label{E:laplace}
\E(e^{-u X_t}) = \exp \left(-x_0 v_t(u) - \int_{v_t(u)}^u \frac{F(z)}{R(z)} dz \right),
\end{equation}
where $t \to v_t(u)$ is the unique locally bounded solution of 
\begin{equation} \label{eq:Laplace Eq Diff}
    \frac{\partial}{\partial t} v_t(u) = - R(v_t(u)), \quad v_0(u)=u.
\end{equation}
For the $\alpha$-CIR process, the branching mechanism and the immigration rate are given by
\[R(z)=\frac{\overline{\delta}^\alpha}{\alpha} z^\alpha +bz, \quad F(z)=az, \quad z \in [0, + \infty),\]
with $\overline{\delta}=\delta (\alpha/| \cos(\frac{\pi \alpha}{2}) |)^{1/ \alpha}$.
Equation \eqref{eq:Laplace Eq Diff} is a Bernoulli differential equation, which we solve using the change of variables $z_t(u)= \frac{1}{v_t(u)^{\alpha-1}}.$ We get the following expressions of $v_t(u)$ (see example 2.3 in Li \cite{Li})
\begin{align}
   \text{if} \;  b \neq 0  \quad &v_t(u) = \frac{u e^{-bt}}{ \left(1+\frac{\overline{\delta}^\alpha}{\alpha b} u^{\alpha-1} \left(1-e^{-(\alpha-1)bt} \right)\right)^{\frac{1}{\alpha-1}}}, \label{eq:vt:b!0}\\
   \text{if} \;  b=0  \quad &v_t(u) = u \left(\frac{\alpha}{\alpha + (\alpha-1) \overline{\delta}^\alpha u^{\alpha-1} t}\right)^\frac{1}{\alpha-1}. \label{eq:vt:b=0}
\end{align}
We have $v_0(u)=u$ and by equation \eqref{eq:Laplace Eq Diff} $t \to v_t(u)$ is non-increasing for $u > u_0$, where
\[ u_0= \left\{ \begin{array}{ll}
        0 & \text{if } b \geq 0, \\
        (\frac{\alpha |b|}{\overline{\delta}^\alpha})^\frac{1}{\alpha-1} & \text{if } b<0.
    \end{array} \right.\]
Hence for $u>u_0$ we have $\forall t>0, u>v_t(u) \geq 0$.

Splitting $\E(1/X_t^q)$ in two parts
\begin{align*}
    \E\left(\frac{1}{X_t^q}\right) &= C_q \int_0^{u_0+1} u^{q-1} \E(e^{-u X_t}) du + C_q \int_{u_0+1}^\infty u^{q-1} \E(e^{-u X_t}) du \\
    &\leq C_q(u_0) + C_q \int_{u_0+1}^\infty u^{q-1} \E(e^{-u X_t}) du,
\end{align*}
we deduce that to prove \eqref{eq:Moments1/X} we only have to check that 
\begin{equation} \label{E:bound}
\sup_{t\in [0,1]} \int_{u_0+1}^\infty u^{q-1} \E(e^{-u X_t}) du < + \infty.
\end{equation}
Let us fix $\eta \in \ ]0, \alpha-1[$. We have
\begin{align*} 
\sup_{t\in [0,1]} \int_{u_0+1}^\infty u^{q-1} \E(e^{-u X_t}) du  & \leq \int_{u_0+1}^\infty u^{q-1} \sup_{t\in [0,\frac{1}{u^\eta}]} \E(e^{-u X_t}) du \\
& + \int_{u_0+1}^\infty u^{q-1} \sup_{t\in [\frac{1}{u^\eta},1]} \E(e^{-u X_t}) du. 
\end{align*}
But
using that $ t \to v_t(u)$ is non-increasing for $u > u_0$, we obtain from \eqref{E:laplace} 
\begin{align*} 
    \sup_{t\in [0, \frac{1}{u^\eta}]} \E(e^{-u X_t}) &\leq \sup_{t\in [0, \frac{1}{u^\eta}]} \exp(-x_0 v_t(u) ) \leq \exp (-x_0 v_{\frac{1}{u^\eta}}(u) ),\\
    \sup_{t \in [\frac{1}{u^\eta},1]} \E(e^{-u X_t}) &\leq \sup_{t \in [\frac{1}{u^\eta},1]} \exp \left(- \int_{v_t(u)}^u \frac{F(z)}{R(z)} dz\right) \\
     & \leq \exp \left( - \left(u - v_{\frac{1}{u^\eta}}(u) \right) \frac{a}{\frac{\overline{\delta}^\alpha}{\alpha}u^{\alpha-1} +b} \right), 
\end{align*}
and then
\begin{align*}
    \sup_{t\in [0,1]} \int_{u_0+1}^\infty u^{q-1} \E(e^{-u X_t}) du & \leq \int_{u_0+1}^\infty u^{q-1} \exp (-x_0 v_{\frac{1}{u^\eta}}(u) ) du \\
    & +  \int_{u_0+1}^\infty u^{q-1} \exp \left[- (u - v_{\frac{1}{u^\eta}}(u) ) \frac{a}{\frac{\overline{\delta}^\alpha}{\alpha}u^{\alpha-1} +b} \right] du.
\end{align*}
From expressions \eqref{eq:vt:b=0} and \eqref{eq:vt:b!0}, we have
$
v_{\frac{1}{u^\eta}}(u) \underset{u \to \infty}{\sim}  C u^\frac{\eta}{\alpha-1},
$
and recalling that $0 < \frac{\eta}{\alpha-1}< 1$ and $\alpha <2$, we get \eqref{E:bound}. \\

\subsection{Proof of Proposition \ref{P:intstoch} }

(i)  The result can be deduced from Proposition 2.7 in \cite{LongNeg}, but we give here a different proof. Using the Markov property
\[ \E \left( \sup_{s\in [\frac{i-1}{n}, \frac{i}{n}]} \left|  \int_\frac{i-1}{n}^s X_{t-}^{1/ \alpha} dL^\alpha_t  \right|^{p} |\mathcal{F}_\frac{i-1}{n}\right) = \E_{x = X_{\frac{i-1}{n}}} \left( \sup_{s\in [0, \frac{1}{n}]} \left|  \int_0^s X_{t-}^{1/ \alpha} dL^\alpha_t  \right|^{p} \right). \]
Using Lemma 2.4 in \cite{Long} with $F(u) = u^p$ and $p < \alpha$
\[ \E_{x} \left( \sup_{s\in [0, \frac{1}{n}]} \left|  \int_0^s X_{t-}^{1/ \alpha} dL^\alpha_t  \right|^{p} \right) \leq C_p \E_{x} \left( \left( \int_0^\frac{1}{n} X_{t-}  dt \right)^{p/ \alpha} \right).  \]
Hence using Proposition \ref{P:moments} equation \eqref{eq:MomentsX}
\[ \E_{x} \left( \sup_{s\in [0, \frac{1}{n}]} \left|  \int_0^s X_{t-}^{1/ \alpha} dL^\alpha_t  \right|^{p} \right) \leq \frac{ C_p}{n^{p/\alpha}} \left( 1+ x^{p / \alpha} \right).\] \\
(ii) Similarly we have
\[ \E \left( \sup_{s\in [\frac{i-1}{n}, \frac{i}{n}]} \left|X_s - X_{\frac{i-1}{n}}\right|^{p}|\mathcal{F}_\frac{i-1}{n}\right) = \E_{x = X_{\frac{i-1}{n}}} \left( \sup_{s\in [0, \frac{1}{n}]} \left|X_s - x\right|^{p}\right). \]
From equation \eqref{eq:EDS} we have that
\[\forall s\in \left[0, \frac{1}{n}\right], \quad \quad  X_s - x = \int_0^s adt - \int_0^s b X_t dt + \delta \int_0^s X_{t-}^{1/ \alpha}dL^\alpha_t.\]
Using Proposition \ref{P:moments} equation \eqref{eq:MomentsX} and Proposition \ref{P:intstoch} (i), we get that
\[ \E_{x} \left( \sup_{s\in [0, \frac{1}{n}]} \left|X_s - x\right|^{p} \right) \leq C_p \left( \frac{1}{n^{p}} \left(1 + x^p \right) + \frac{1}{n^{p / \alpha}} \left(1 + x^{p / \alpha} \right) \right) \leq \frac{C_p}{n^{p / \alpha}} \left(1 + x^{p } \right).\] \\
(iii) 
\[ \E\left( \left|  \int_\frac{i-1}{n}^\frac{i}{n} (X_{s-}^{1/ \alpha} - X_{\frac{i-1}{n}}^{1/ \alpha}) dL^\alpha_s  \right|^{p} |\mathcal{F}_\frac{i-1}{n}\right) = \E_{x = X_{\frac{i-1}{n}}} \left( \left|  \int_0^\frac{1}{n} (X_{s-}^{1/ \alpha} - x^{1/ \alpha}) dL^\alpha_s  \right|^{p} \right). \]
Using Lemma 2.8 in \cite{LongNeg} with $p < \alpha$
\[ \E_{x} \left( \left|  \int_0^\frac{1}{n} (X_{s-}^{1/ \alpha} - x^{1/ \alpha}) dL^\alpha_s  \right|^{p} \right) \leq C_p \E_{x} \left( \left(  \int_0^\frac{1}{n} \left|X_{s-}^{1/ \alpha} - x^{1/ \alpha}\right|^{\alpha}  ds \right)^{p/ \alpha} \right).  \]
But from Taylor's Formula, we have
\begin{equation} \label{x1/a - y1/a}
    \forall x,y \in ]0, +\infty[,  \quad | y^{1/ \alpha} - x^{1/ \alpha} |  = \frac{|y-x|}{\alpha} \int_0^1 ((1-u)x+uy)^{1/ \alpha-1} du\leq \frac{C |y-x|}{x^{1 - 1/\alpha}},
\end{equation}
and we obtain
\begin{align*}
    \E_{x} \left( \left(  \int_0^\frac{1}{n} \left|X_{s-}^{1/ \alpha} - x^{1/ \alpha}\right|^{\alpha}  ds \right)^{p/ \alpha} \right) 
    & \leq \frac{C_p}{x^{p - p/ \alpha}} \frac{1}{n^{p/ \alpha}}\E_{x} \left( \sup_{s\in [0, \frac{1}{n}]} \left|X_s - x\right|^{p} \right).
\end{align*}
Hence combining with (ii) it yields
\[ \E_{x} \left( \left|  \int_0^\frac{1}{n} (X_{s-}^{1/ \alpha} - x^{1/ \alpha}) dL_s  \right|^{p} \right) \leq \frac{C_p}{n^{2p/ \alpha}} \left( 1+ \frac{1}{x^p} + x^p\right).\]


\subsection{Estimation of $z^n$} \label{Ss:zn}

We prove in this section that the rescaled increment $z_i^n(\theta)$ defined by \eqref{eq:def zn} converges to $n^{1/ \alpha_0} \Delta_i^n L$. 
We also give a bound for the conditional expectation of $|z_i^n(\theta)|^q {\bf 1}_{\{ z_i^n(\theta)<0\} }$. We recall that $W_n^{(\eta)}$ is defined by \eqref{Voisinage Wn}.
\begin{lem}  \label{L:zn-L1}
Let $0<p<\alpha_0$. 
    Then for $A \subset (0, + \infty) \times \R$ a compact set we have for some constant $q>0$
    \begin{eqnarray} \label{eq:sup zn-L1 unif}
        \E \left(\sup_{(a,b) \in A, \; (\delta, \alpha) \in W_n^{(\eta)} } \left|z^n_i(\theta) - n^{1/\alpha_0} \Delta_i^n L \right|^p  |\mathcal{F}_\frac{i-1}{n}\right)  \quad \quad \quad \quad \quad\nonumber \\
\quad \quad \quad \quad         \leq C_p \left( 1 + \frac{1}{X_{\frac{i-1}{n}}^{q}} + X_{\frac{i-1}{n}}^{q} \right) \left( \frac{n^{p/ \alpha_0}}{n^{p}} 
    + (\ln (n) w_n)^p \right),
    \end{eqnarray}
    and for the true value $\theta_0$ we have the better rate of convergence
    \begin{equation} \label{eq:zn(theta0)-L1}
        \E\left(\left|z^n_i(\theta_0) - n^{1/\alpha_0} \Delta_i^n L \right|^p | \mathcal{F}_\frac{i-1}{n} \right)  \leq \frac{C_p}{n^{p/ \alpha_0}} \left( 1 + \frac{1}{X_{\frac{i-1}{n}}^{q}} + X_{\frac{i-1}{n}}^p \right).
        \end{equation}
       \end{lem}

\begin{proof}
We first prove \eqref{eq:sup zn-L1 unif}.   Using the Markov property, we have to bound
\[  \E_{x = X_{\frac{i-1}{n}}}\sup_{(a,b) \in A, (\delta, \alpha) \in W_n^{(\eta)} } \left|z^n_1(\theta) - n^{1/\alpha_0} L_{\frac{1}{n}} \right|^p .\]
Since $X$ solves \eqref{eq:EDS} for the parameter value $\theta_0$, we have conditionally to $X_0=x$
\begin{equation} \label{eq:formule zn}
    z_1^n(\theta)= \frac{n^{1/\alpha}}{\delta x^{1/ \alpha}} \left( \frac{1}{n} (a_0-a) - \int_0^\frac{1}{n} (b_0 X_s - bx )ds + \delta_0 \int_0^\frac{1}{n} X_{s-}^{1/ \alpha_0} dL_s \right).
    \end{equation}
Observing that $z_1^n(\theta) - n^{1/\alpha_0} L_{\frac{1}{n}} =z^n_1(\theta) - \frac{\delta_0}{\delta}n^{1/\alpha_0} L_{\frac{1}{n}} + \frac{\delta_0 - \delta}{\delta}n^{1/\alpha_0} L_{\frac{1}{n}}$, we deduce
\begin{align*}
z_1^n(\theta) - n^{1/\alpha_0} L_{\frac{1}{n}} = &
\frac{n^{1/\alpha} }{n } \frac{(a_0-a)}{\delta x^{1/ \alpha}} 
- \frac{n^{1/\alpha} }{n } \frac{(b_0-b) x}{\delta x^{1/ \alpha}} 
- \frac{n^{1/\alpha}}{\delta x^{1/ \alpha}} b_0 \int_0^\frac{1}{n} \left(X_s - x \right) ds  \\
& +  \frac{\delta_0 }{\delta } \frac{n^{1/\alpha_0}}{x^{1/ \alpha_0}} \int_0^\frac{1}{n} (X_{s-}^{1/ \alpha_0} - x^{1/ \alpha_0}) dL_s \\
&
- \frac{\delta_0 }{\delta } \frac{n^{1/\alpha_0}}{x^{1/ \alpha_0}} \left(1 - \frac{n^{1/ \alpha - 1/ \alpha_0}}{x^{1/ \alpha - 1/ \alpha_0}} \right) \int_0^\frac{1}{n} X_{s-}^{1/ \alpha_0} dL_s  + \frac{\delta_0 - \delta}{\delta}n^{1/\alpha_0} L_{\frac{1}{n}}.
\end{align*}
For $(\delta, \alpha) \in W_n^{(\eta)}$, we have $| \delta-\delta_0 | \leq C w_n$ and  $| \alpha-\alpha_0 | \leq C w_n$. Moreover using Taylor's formula
\begin{align*}
    \left|1 - \frac{n^{1/ \alpha - 1/ \alpha_0}}{x^{1/ \alpha - 1/ \alpha_0}} \right| 
    & = \left| \ln(n/x)(1/ \alpha - 1/ \alpha_0) \right| e^y, \; \text{with} \;  |y| \leq |\ln(n/x)(1/ \alpha - 1/ \alpha_0)|. 
\end{align*}    
   We have $| \ln(n/x)(1/ \alpha - 1/ \alpha_0)| \leq C  \ln(n) \left( 1 + x +  \frac{1}{x} \right)w_n$ and using that $w_n \ln(n)$ goes to zero, we also observe that $e^{| y|} \leq Ce^{c| \ln(x)|} \leq C( 1 + x^c +  \frac{1}{x^c} )$ for $c>0$. We then deduce 
\begin{align}
    \left|1 - \frac{n^{1/ \alpha - 1/ \alpha_0}}{x^{1/ \alpha - 1/ \alpha_0}} \right| 
    & \leq C  \left( 1 + x^q +  \frac{1}{x^q} \right)\ln(n)w_n \quad \text{for} \quad q>0.   \label{pr distance alpha alpha0}
\end{align}
It yields
\begin{align*}
    \sup_{(a,b) \in A, \; (\delta, \alpha) \in W_n^{(\eta)} } \left|z^n_1(\theta) - n^{1/\alpha_0} L_{1 /n} \right| 
    \leq C  \left( 1 + x^q + \frac{1}{x^q} \right) \left[\frac{n^{1/\alpha_0} }{n }  \right.\\
   +  \frac{n^{1/\alpha_0} }{n } \sup_{s\in [0, \frac{1}{n}]} \left|X_s - x\right| 
    +  n^{1/\alpha_0}  | \int_0^\frac{1}{n} (X_{s-}^{1/ \alpha_0} - x^{1/ \alpha_0}) dL_s  | \\
 \left.   + \ln(n) w_n n^{1/\alpha_0} |\int_0^\frac{1}{n} X_{s-}^{1/ \alpha_0} dL_s| \right]
    + C w_n n^{1/\alpha_0} |L_{\frac{1}{n}}|.
\end{align*}
Combining this result with Proposition \ref{P:intstoch}, and using  that
$ \E( | n^{1/\alpha_0} L_{1 /n}|^p ) = \E (|L_1|^p) < + \infty$, for $p \in (0, \alpha_0)$,
we deduce finally for some $q>0$
\begin{align*}
    \E_{x}  \sup_{(a,b) \in A, \; (\delta, \alpha) \in W_n^{(\eta)} } \left|z^n_1(\theta) - n^{1/\alpha_0} L_{1 /n} \right|^p   & \leq C_p ( 1 + \frac{1}{x^{q}} + x^{q} ) \left[ \frac{n^{p/\alpha_0}}{n^p} + \frac{1}{n^{p/ \alpha_0}}  \right. \\
    & \left. + (\ln (n)w_n)^p  \right],
\end{align*}
and \eqref{eq:sup zn-L1 unif} is proved.

Turning to \eqref{eq:zn(theta0)-L1}, we  
  have the simpler decomposition conditionally to $X_0=x$
\begin{align*}
z_1^n(\theta_0) - n^{1/\alpha_0} L_{1 /n} = -b_0\frac{n^{1/\alpha_0}}{\delta_0 x^{1/ \alpha_0}} \int_0^\frac{1}{n} \left(X_s - x \right) ds  + \frac{n^{1/\alpha_0}}{x^{1/ \alpha_0}} \int_0^\frac{1}{n} (X_{s-}^{1/ \alpha_0} - x^{1/ \alpha_0}) dL_s.
\end{align*}
Using  Proposition \ref{P:intstoch}
 we obtain
\begin{align*}
    \E_{ x} \left| z_1^n(\theta_0) - n^{1/\alpha_0} L_{1 /n} \right|^p  \leq \frac{C_p}{n^{p/ \alpha_0}} \left( 1 + \frac{1}{x^{q}} + x^p \right).
\end{align*}
\end{proof}

\begin{lem}  \label{L:Esp(zn puissance k)}
Let $\Omega_K=\{ \sup_{s \in [0,1]} X_s \leq K \}$ and $W_n^{(\eta)}$ be defined by \eqref{Voisinage Wn}. We have for any compact set $A \subset (0, +\infty) \times \R$ and $ \forall q > 0$
\begin{equation*} 
 \E  \left(\sup_{(a,b) \in A, \; (\delta, \alpha) \in W_n^{(\eta)} } \left(|z^n_i(\theta) |^q {\bf 1}_{z^n_i(\theta)<0} \right) {\bf 1}_{\Omega_K} |\mathcal{F}_\frac{i-1}{n}\right)  \leq C_{q,K} \left( 1 + \frac{1}{X_{\frac{i-1}{n}}^{q}} \right).
\end{equation*}
\end{lem}

\begin{proof}
We introduce the localisation
$\Omega_K^{i,n} = \{ \sup_{s \in [\frac{i-1}{n}, \frac{i}{n}]} X_s \leq K \}$.  Since $\Omega_K  \subset \Omega_K^{i,n}$, we just have to bound
\[
\E\left( \sup_{(a,b) \in A,\;  (\delta, \alpha) \in W_n^{(\eta)} } \left(|z^n_i(\theta) |^q {\bf 1}_{z^n_i(\theta)<0} \right) {\bf 1}_{\Omega_K^{i,n}} |\mathcal{F}_\frac{i-1}{n}\right),\]
and from  the Markov property, this reduces to bound
\[
\E_{x=X_{\frac{i-1}{n}}} \left( \sup_{(a,b) \in A,\;  (\delta, \alpha) \in W_n^{(\eta)} } \left(|z^n_1(\theta) |^q {\bf 1}_{z_1^n(\theta)<0} \right) {\bf 1}_{\Omega_K^{1,n}} \right).\]
From equation \eqref{eq:formule zn}, conditionally to $X_0=x$, we can write 
\[z_1^n(\theta) = \frac{1}{ x^{1/ \alpha}} \left( Y^{\theta ,n} + \frac{\delta_0}{\delta} n^{1/\alpha - 1/\alpha_0} I_1^n \right),\] 
where
\[Y^{\theta ,n} = \frac{n^{1/\alpha}}{\delta} \left( \frac{1}{n}(a_0-a) - \int_0^\frac{1}{n} (b_0 X_s - b x )ds \right), \]
and
\[ I_1^n  = n^{1/\alpha_0} \int_0^\frac{1}{n} X_{s-}^{1/ \alpha_0} dL_s. \]
Observing that 
\[  \sup_{(a,b) \in A,\;  (\delta, \alpha) \in W_n^{(\eta)} }|Y^{\theta ,n}| {\bf 1}_{\Omega_K^{1,n}} \leq C_K,\]
 there exists a constant $M_K>0$ such that
\[ \left\{z^n_1(\theta) \leq \frac{-M_K}{x^{1/ \alpha}} \right\} \cap \Omega_K^{1,n} \subset \left\{ I_1^n  < 0 \right\} \cap \Omega_K^{1,n} .\]
Using the decomposition
\[{\bf 1}_{z_1^n(\theta)<0} = {\bf 1}_{\left\{z^n_1(\theta) \leq \frac{-M_K}{x^{1/ \alpha}} \right\}} + {\bf 1}_{\left\{ \frac{-M_K}{x^{1/ \alpha}} <z^n_1(\theta) < 0 \right\}},\]
we deduce then
\[ \sup_{(a,b) \in A, \; (\delta, \alpha) \in W_n^{(\eta)} } \left(|z^n_1(\theta) |^q {\bf 1}_{z_1^n(\theta)<0} \right) {\bf 1}_{\Omega_K^{1,n}}  \leq C_K \left( 1 + \frac{1}{x^{q}} \right) \left(1 + |I_1^n|^q  {\bf 1}_{\left\{ I_1^n  < 0 \right\}} \right){\bf 1}_{\Omega_K^{1,n}} . \]
Consequently, we just have to bound $\E_{x}  |I_1^n|^q {\bf 1}_{\left\{ I_1^n  < 0 \right\}} {\bf 1}_{\Omega_K^{1,n}} $.
To this end, we first remark that from Fubini's Theorem we have
\begin{equation} \label{eq : I ni}
    \E_{x} \left( |I_{1}^n|^q {\bf 1}_{\left\{ I_1^n  < 0 \right\}} {\bf 1}_{\Omega_K^{1,n}} \right) = \int_0^{+ \infty} q s^{q-1} \PP_{x} \left( \Omega_K^{1,n} \cap \{ I_1^n <-s\}\right) ds.
\end{equation}
Using the change of time described in Theorem 1.5 of \cite{ChangementTemps}, that we can apply as $\forall t > 0, X_t \geq 0$, we have the representation
\[ I_1^n = \tilde{L}_{\hat{T}_{n,1}}, \quad \text{where } \hat{T}_{n,1} = \int_0^\frac{1}{n} X_s  ds , \]
and $(\tilde{L}_t)_{t \geq 0}$ is a $\alpha_0$-stable L\'evy process (with characteristic function \eqref{E:char}) defined on the same probability space. 
By definition of $\Omega_K^{1,n}$
\[ 0 \leq {\bf 1}_{\Omega_K^{1,n}} \hat{T}_{n,1} \leq K,\]
then
\[ \forall s > 0, \quad \PP_x \left( \Omega_K^{1,n} \cap \{I_1^n < -s\} \right) \leq \PP_x \left( \inf_{t \in [0,K]} \tilde{L}_t < -s \right),\]
and from Lemma 2.4 in \cite{ProbaAlphaStableNegative} we obtain
\begin{align} \label{eq : alpha-stable neg}
    \forall s > 0, \quad \PP_x \left( \Omega_K^{1,n} \cap \{I_1^n < -s\} \right) \leq \exp \left( - \left( \frac{\alpha_0-1}{\alpha_0} \right)^{\frac{\alpha_0}{\alpha_0 -1}} \frac{s^{\frac{\alpha_0}{\alpha_0-1}}}{K^{\frac{1}{\alpha_0-1}} }\right).
\end{align}
Reporting \eqref{eq : alpha-stable neg} in \eqref{eq : I ni}, it yields
\[ \E_{x} \left( |I_{1}^n|^q {\bf 1}_{\left\{ I_1^n  < 0 \right\}} {\bf 1}_{\Omega_K^{1,n}} \right) \leq C_{q,K}.\]
This allows to conclude that
\[ \E_{x} \left( \sup_{(a,b) \in A, \; (\delta, \alpha) \in W_n^{(\eta)} } \left(|z^n_1(\theta) |^q {\bf 1}_{z_1^n(\theta)<0} \right) {\bf 1}_{\Omega_K^{1,n}} \right) \leq C_{q,K} \left( 1 + \frac{1}{x^{q}} \right) .\]

\end{proof}

\subsection{Proof of Theorem \ref{Th : LFGN}} \label{Ss:LFGN}

We first remark that if $h$ satisfies Assumption \ref{A:conditions h} then using Taylor's formula we have  $\forall \alpha, \alpha' \in (1,2)$ and $ \forall x, z \in \R$
    \begin{align} \label{E:Rmk h(x) - h(z)}
       \left|h (x, \alpha) - h (z, \alpha)\right| \leq  & C \left|x - z \right| \left( 1 + |x|^q {\bf 1}_{x<0} + (\ln(1+x))^q {\bf 1}_{x>0} \right. \nonumber  \\
        & \quad \quad  \left. + |z|^q {\bf 1}_{z<0} + (\ln(1+z))^q {\bf 1}_{z>0} \right),
    \end{align}
    \begin{equation*}
        \left|h (x, \alpha) - h (x, \alpha')\right| \leq C \left|\alpha - \alpha' \right| \left( 1 + |x|^q {\bf 1}_{x<0} + (\ln(1+x))^q {\bf 1}_{x>0}\right).
    \end{equation*}
To prove the result of Theorem \ref{Th : LFGN}, we check the following convergences in probability

\begin{equation}\label{pr:thet-thet0}
    \sup_{ \theta \in A \times W_n^{(\eta)} } (\ln n)^q\left| \frac{1}{n}\sum_{i=1}^n \left( f(X_{\frac{i-1}{n}}, \delta, \alpha) - f(X_{\frac{i-1}{n}}, \delta_0, \alpha_0) \right) h(z^n_i(\theta), \alpha) \right| \to 0,
\end{equation}
\begin{equation}\label{pr:zn-L1}
    \sup_{\theta \in A \times W_n^{(\eta)} }  (\ln n)^q \left| \frac{1}{n}\sum_{i=1}^n f(X_{\frac{i-1}{n}}, \delta_0, \alpha_0) \left( h(z^n_i(\theta), \alpha) - h(n^{1/ \alpha_0} \Delta_i^n L,\alpha) \right) \right| \to 0,
\end{equation}
\begin{equation}\label{pr:alph-alph0}
    \sup_{\theta \in A \times W_n^{(\eta)} }  (\ln n)^q \left| \frac{1}{n}\sum_{i=1}^n f(X_{\frac{i-1}{n}}, \delta_0, \alpha_0) \left( h(n^{1/ \alpha_0} \Delta_i^n L,\alpha) - h(n^{1/ \alpha_0} \Delta_i^n L, \alpha_0) \right) \right| \to 0,
\end{equation}
\begin{equation} \label{pr:L1-esp}
 ( \ln n)^q   \left| \frac{1}{n}\sum_{i=1}^n f(X_{\frac{i-1}{n}}, \delta_0, \alpha_0) \left( h(n^{1/ \alpha_0} \Delta_i^n L,\alpha_0) - \E(h(L_1,\alpha_0)) \right) \right| \to 0,
\end{equation}
\begin{equation}\label{pr:f-int}
  (\ln n)^q  \left| \frac{1}{n}\sum_{i=1}^n f(X_{\frac{i-1}{n}}, \delta_0, \alpha_0) - \int_0^1 f(X_s, \delta_0, \alpha_0)ds \right| \to 0.    
\end{equation}
To simplify the presentation, in the proof of \eqref{pr:zn-L1}-\eqref{pr:f-int}, we omit the dependence in $(\delta_0, \alpha_0)$ in the expression of $f$.

\noindent
\underline{Proof of \eqref{pr:f-int}.}  First we have  
$\int_0^1 \left| f(X_s) \right| ds < + \infty \quad \text{a.e.}$
as a consequence of Proposition \ref{P:moments} and as $s \to X_s$ is c\`adl\`ag. Next
we set
\[I_n = (\ln n)^q\left|\frac{1}{n} \sum_{i=1}^n f(X_{\frac{i-1}{n}}) -  \int_0^1 f(X_s) ds \right|  .\]
Introducing the localisation $\Omega_K=\{ \sup_{s \in [0,1]} X_s \leq K \}$, we just have to prove the convergence to zero of  
 $I_n {\bf 1}_{\Omega_K} $, $\forall \ K>0$. But from Taylor's formula
\begin{align*}
    I_n {\bf 1}_{\Omega_K} & \leq (\ln n)^q\left| \int_0^1 \left( f(X_s) - f(X_\frac{\floor{ns}}{n}) \right) ds \right|  {\bf 1}_{\Omega_K}  \\
    & \leq C_K (\ln n)^q\int_0^1 \left| X_s - X_\frac{\floor{ns}}{n} \right| \left( 1 + \frac{1}{X_s^q} + \frac{1}{X_\frac{\floor{ns}}{n}^q} \right) ds.
\end{align*}
Therefore using successively Fubini's Theorem and H\"{o}lder's inequality with $1<p<\alpha_0$
\[\E \left( I_n {\bf 1}_{\Omega_K} \right) \leq C_K (\ln n)^q\int_0^1 \left(\E \left| X_s - X_\frac{\floor{ns}}{n} \right|^p \right)^{1/p} \left(\E  \left( 1 + \frac{1}{X_s^q} + \frac{1}{X_\frac{\floor{ns}}{n}^q} \right)^{p'} \right)^{1/p'} ds.\]
From Proposition \ref{P:intstoch} and  Proposition \ref{P:moments}, we get
$\E \left( I_n {\bf 1}_{\Omega_K} \right) \leq  C_{p,K} (\ln n)^q/n^{1 / \alpha_0}$ and we deduce the result.

\noindent
\underline{Proof of \eqref{pr:L1-esp}.}  We recall classical results on the convergence in probability of triangular arrays (see \cite{TriangularArray}). Let $(\zeta_i^n)$ be a triangular array such that $\zeta_i^n$ is $\mathcal{F}_\frac{i-1}{n}$-measurable. To prove that $\sum_{i=1}^n \zeta_i^n \to 0$ in probability, it is sufficient to check the following convergences in probability
\begin{align*}
    \sum_{i=1}^n \left|\E ( \zeta_i^n  |\mathcal{F}_\frac{i-1}{n})\right| \to 0 \quad \text{ and} \quad
    \sum_{i=1}^n \E( \left|\zeta_i^n \right|^2 |\mathcal{F}_\frac{i-1}{n}) \to 0.
\end{align*}
Setting 
\[\zeta_i^n = \frac{(\ln n)^q}{n} f(X_{\frac{i-1}{n}}) \left( h(n^{1/ \alpha_0} \Delta_i^n L,\alpha_0) - \E(h(L_1,\alpha_0)) \right),\]
we check immediately
$\E(  \zeta_i^n |\mathcal{F}_\frac{i-1}{n}) = 0.$
Next we deduce from Assumption \ref{A:conditions h}  
\[ \forall p > 0, \quad \E\left( |h(L_1,\alpha_0)|^p \right) < +\infty, \]
and it yields
\[ \E( \left|\zeta_i^n \right|^2 |\mathcal{F}_\frac{i-1}{n}) = \frac{(\ln n)^{2q}}{n^2} f(X_{\frac{i-1}{n}})^2 \text{Var}(h(L_1,\alpha_0)).\]
We conclude using  the convergence in probability  \ref{pr:f-int}
\[\frac{1}{n} \sum_{i=1}^n f(X_{\frac{i-1}{n}})^2 \to \int_0^1 f(X_s)^2 ds.\]
\noindent
\underline{Proof of \eqref{pr:thet-thet0}.} We use as previously the truncation $\Omega_K$ and set
$$
T_n=\sup_{(a,b) \in A, \; (\delta, \alpha) \in W_n^{(\eta)} } \biggl| \frac{(\ln n)^q}{n}\sum_{i=1}^n \biggl( f(X_{\frac{i-1}{n}}, \delta, \alpha) - f(X_{\frac{i-1}{n}}, \delta_0, \alpha_0) \biggl) h(z^n_i(\theta),\alpha) \biggl| {\bf 1}_{\Omega_K}.  
$$
A Taylor expansion, assumption \eqref{eq : conditions f}, and the definition of $W_n^{(\eta)}$ give 
\begin{align*}
T_n 
\leq C_K \frac{(\ln n)^qw_n}{n}\sum_{i=1}^n \ \left( 1 + \frac{1}{X_{\frac{i-1}{n}}^q}\right) \sup_{(a,b) \in A, \; (\delta, \alpha) \in W_n^{(\eta)} } \left| h(z^n_i(\theta),\alpha) \right| {\bf 1}_{\Omega_K},&
\end{align*}
and taking the expectation
\begin{align*}
    \E T_n    \leq  C_K  \frac{(\ln n)^qw_n}{n}\sum_{i=1}^n 
    \E  \left( (1 + \frac{1}{X_{\frac{i-1}{n}}^p}) \E \left( \sup_{(a,b) \in A, \; (\delta, \alpha) \in W_n^{(\eta)} } \left| h(z^n_i(\theta),\alpha) \right| {\bf 1}_{\Omega_K} |\mathcal{F}_\frac{i-1}{n} \right) \right).
\end{align*}
But from Assumption  \ref{A:conditions h} 
\begin{align*}
    \sup_{(a,b) \in A, \; (\delta, \alpha) \in W_n^{(\eta)} }  \left| h(z^n_i(\theta),\alpha) \right|{\bf 1}_{\Omega_K} 
& \leq C \left( 1 + \sup_{(a,b) \in A, \; (\delta, \alpha) \in W_n^{(\eta)} } \left|z^n_i(\theta) \right|^{q} {\bf 1}_{z^n_i(\theta)<0}  \right.\\
 & \left. + \sup_{(a,b) \in A, \; (\delta, \alpha) \in W_n^{(\eta)} } (\ln(1+z^n_i(\theta)))^{q} {\bf 1}_{z^n_i(\theta)>0} \right) {\bf 1}_{\Omega_K}.
\end{align*}
Observing from \eqref{eq:def zn} that 
\begin{equation} \label{E:logzn}
 \sup_{(a,b) \in A, \; (\delta, \alpha) \in W_n^{(\eta)} } \ln(1+|z^n_i(\theta) |) {\bf 1}_{\Omega_K} \leq C_K \ln(n) \left(1 + \frac{1}{X_{\frac{i-1}{n}}} \right),
 \end{equation}
and using Lemma \ref{L:Esp(zn puissance k)}, we deduce
\begin{align*}
    \E\left( \sup_{(a,b) \in A, \; (\delta, \alpha) \in W_n^{(\eta)} } \left| h(z^n_i(\theta),\alpha) \right| {\bf 1}_{\Omega_K}  |\mathcal{F}_\frac{i-1}{n}\right)
    \leq C_K (\ln n)^q \left(1+ \frac{1}{X_{\frac{i-1}{n}}^q} \right) .
\end{align*}
Hence from Proposition \ref{P:moments} $\E T_n \leq C_K (\ln n)^q w_n$ 
and we conclude using \eqref{Voisinage Wn}.

\noindent
\underline{Proof of \eqref{pr:alph-alph0}.} We get from   Assumption \ref{A:conditions h}, \eqref{eq : conditions f},  \eqref{E:Rmk h(x) - h(z)} and the definition of $W_n^{(\eta)}$
\begin{align*}
    \sup_{(a,b) \in A, \; (\delta, \alpha) \in W_n^{(\eta)} } & \left|  \frac{(\ln n)^q}{n}\sum_{i=1}^n f(X_{\frac{i-1}{n}}) \left( h(n^{1/ \alpha_0} \Delta_i^n L,\alpha) - h(n^{1/ \alpha_0} \Delta_i^n L,\alpha_0) \right) {\bf 1}_{\Omega_K} \right| \\
    & \leq C_K \frac{(\ln n)^qw_n}{n}\sum_{i=1}^n ( 1 + \frac{1}{X_{\frac{i-1}{n}}^q}) \left[ 1 + |n^{1/ \alpha_0} \Delta_i^n L|^q {\bf 1}_{n^{1/ \alpha_0} \Delta_i^n L<0} \right. \\
& \quad \quad \quad \quad \left.     + (\ln(1+n^{1/ \alpha_0} \Delta_i^n L))^q {\bf 1}_{n^{1/ \alpha_0} \Delta_i^n L>0} \right].
\end{align*}
Taking the expectation,  using the properties of the L\'evy process $(L_t)$ and Proposition \ref{P:moments}, we deduce the result. 

\noindent
\underline{Proof of \eqref{pr:zn-L1}.} Using Assumption  \ref{A:conditions h},  \eqref{eq : conditions f} and  \eqref{E:Rmk h(x) - h(z)}

\begin{align} \label{E:zn-L1}
    \E \biggl( \sup_{(a,b) \in A, \; (\delta, \alpha) \in W_n^{(\eta)} } \biggl| \frac{(\ln n)^q}{n}\sum_{i=1}^n f(X_{\frac{i-1}{n}}) \left( h(z^n_i(\theta),\alpha) - h(n^{1/ \alpha_0} \Delta_i^n L,\alpha) \right) \biggl| {\bf 1}_{\Omega_K} \biggl) \nonumber \\
     \leq C_K \frac{(\ln n)^q}{n}\sum_{i=1}^n \E \left(
    \sup_{(a,b) \in A, \; (\delta, \alpha) \in W_n^{(\eta)} } \left|z^n_i(\theta) - n^{1/\alpha_0} \Delta_i^n L \right| {\bf 1}_{\Omega_K} H_{i,n} \right),
\end{align}
where
\begin{align*}
H_{i,n} = ( 1 + \frac{1}{X_{\frac{i-1}{n}}^q} ) 
\left[ 1 + \sup_{(a,b) \in A, \; (\delta, \alpha) \in W_n^{(\eta)} } \left(|z^n_i(\theta) |^q {\bf 1}_{z^n_i(\theta)<0} {\bf 1}_{\Omega_K} \right)  \right.\\
+ \sup_{(a,b) \in A, \; (\delta, \alpha) \in W_n^{(\eta)} } \left(\ln(1+z^n_i(\theta))^q {\bf 1}_{z^n_i(\theta)>0} {\bf 1}_{\Omega_K} \right) \\
\left.  + \left|n^{1/ \alpha_0} \Delta_i^n L \right|^q {\bf 1}_{n^{1/ \alpha_0} \Delta_i^n L<0}  + \left(\ln (1+ |n^{1/ \alpha_0} \Delta_i^n L | )\right)^q {\bf 1}_{n^{1/ \alpha_0} \Delta_i^n L > 0} \right].
\end{align*}
We use H\"{o}lder's inequality with $\frac{1}{p} + \frac{1}{p'} = 1$ and $p < \alpha_0$ to bound \eqref{E:zn-L1}. From Lemma \ref{L:Esp(zn puissance k)}, Proposition \ref{P:moments} and \eqref{E:logzn} we have immediately
$$
 \E \left( H_{i,n}^{p'} \right) \leq C_{p',K} (\ln n)^{p'}.
$$
Moreover since $\Omega_K \subset \{ X_{\frac{i-1}{n}} \leq K \}$, we deduce combining Lemma \ref{L:zn-L1} equation \eqref{eq:sup zn-L1 unif} with Proposition \ref{P:moments}
\begin{align*}
    \E \left(\sup_{(a,b) \in A, \; (\delta, \alpha) \in W_n^{(\eta)} } \left|z^n_i(\theta) - n^{1/ \alpha_0} \Delta_i^n L \right|^p {\bf 1}_{\Omega_K} \right)      \leq C_{p,K}  (\frac{n^{p/ \alpha_0}}{n^{p}}  + (\ln (n) w_n)^p). 
    \end{align*}
Since $\forall q>0$ $(\ln n)^q w_n \rightarrow 0$, we obtain the convergence to zero of the right-hand side term of \eqref{E:zn-L1}.


\subsection{Proof of Theorem \ref{Th : TCL}} \label{Ss:TCL}
    
    We first prove the following convergence in probability for $k, j \in \{1, ... , d \}$
\begin{equation}
    \frac{1}{\sqrt{n}} \sum_{i=1}^n f_{kj}(X_{\frac{i-1}{n}}) \left( h_j(z^n_i(\theta_0)) - h_j( n^{1/ \alpha_0} \Delta_i^n L) \right) \to 0.
\end{equation}
Using  Assumption \ref{A:conditions h},  \eqref{eq : conditions f i} and  \eqref{E:Rmk h(x) - h(z)} with the truncation $\Omega_K$
\begin{align*}
     \E \biggl( \biggl| \frac{1}{\sqrt{n}}\sum_{i=1}^n f_{kj}(X_{\frac{i-1}{n}}) & \left( h_j(z^n_i(\theta_0)) - h_j(n^{1/ \alpha_0} \Delta_i^n L) \right) \biggl| {\bf 1}_{\Omega_K} \biggl) \\
    & \leq C_K \frac{1}{\sqrt{n}} \sum_{i=1}^n 
 \E \left( \left|z^n_i(\theta _0) - n^{1/\alpha_0} \Delta_i^n L \right| {\bf 1}_{\Omega_K} H_{i,n,0} \right) ,
\end{align*}
where
\begin{align*}
H_{i,n,0} = & \left( 1 + \frac{1}{X_{\frac{i-1}{n}}^q} \right)  \biggl( 1 + |z^n_i(\theta_0) |^q {\bf 1}_{z^n_i(\theta_0)<0} {\bf 1}_{\Omega_K}  + \ln(1+z^n_i(\theta_0))^q {\bf 1}_{z^n_i(\theta_0)>0} {\bf 1}_{\Omega_K}  \\
& + |n^{1/ \alpha_0} \Delta_i^n L |^q {\bf 1}_{n^{1/ \alpha_0} \Delta_i^n L<0}  + \ln (1+ |n^{1/ \alpha_0} \Delta_i^n L | )^q {\bf 1}_{n^{1/ \alpha_0} \Delta_i^n L > 0} \biggl).
\end{align*}
We use H\"{o}lder's inequality with $\frac{1}{p} + \frac{1}{p'} = 1$ and $p < \alpha_0$. As previously, we have
$$
 \E \left( H_{i,n,0}^{p'} \right) \leq C_{p',K} (\ln n)^{p'},
$$
and from Lemma \ref{L:zn-L1} equation \eqref{eq:zn(theta0)-L1} 
\begin{align*}
    \E \left(\left|z^n_i(\theta_0) - n^{1/ \alpha_0} \Delta_i^n L \right|^p {\bf 1}_{\Omega_K} \right) 
    \leq C_{p,K} \frac{1}{n^{p/\alpha_0}}.
     \end{align*}
Consequently  we obtain 
\[ \E \left( \left| \frac{1}{\sqrt{n}}\sum_{i=1}^n f_{kj}(X_{\frac{i-1}{n}}, \theta_0) \left( h_j(z^n_i(\theta_0)) - h_j(n^{1/ \alpha_0} \Delta_i^n L) \right) \right| {\bf 1}_{\Omega_K} \right) \leq C_{p,K}  \frac{\sqrt{n} \ln n}{ n^{1/\alpha_0}}.\]
The results follows from $ \frac{1}{\alpha_0} > \frac{1}{2}$.

We now show that 
\begin{equation}
    \frac{1}{\sqrt{n}} \sum_{i=1}^n F(X_{\frac{i-1}{n}}) H(n^{1/ \alpha_0} \Delta_i^n L)  \xrightarrow[n \to \infty]{\mathcal{L}-s} \Sigma^{1/2} \mathcal{N}.
\end{equation}
We will prove the stable convergence in law with respect to $\sigma(L_s, s \leq 1)$ of the process
\[\Gamma_t^n=
    \frac{1}{\sqrt{n}} \sum_{i=1}^{\floor{nt}} F(X_{\frac{i-1}{n}}) H(n^{1/ \alpha_0} \Delta_i^n L), \quad t \in [0,1],\]
in $\mathbb{D}([0, 1], \mathbb{R}^d)$ equipped with the Skorokhod topology, following the proof of Theorem 3.2 in \cite{ClementEstimating1}. To this end we introduce the processes
\[\overline{L}_t^n= \sum_{i=1}^{\floor{nt}} \Delta_i^n L, \quad t \in [0,1],\]
\[\Gamma'^{n}_t=
    \frac{1}{\sqrt{n}} \sum_{i=1}^{\floor{nt}} 
        H(n^{1/ \alpha_0} \Delta_i^n L), \quad t \in [0,1].\]
The process $(\overline{L}_t^n)_t$ converges in probability to $(L_t)_t$ for the Skorokhod topology and according to Lemma 2.8 in \cite{JacodTCL}, if $(\overline{L}_1^n, \Gamma'^{n}_1)$ converges in law to $(L_1, \gamma')$ where $\gamma'$ is a Gaussian variable independent of $L_1$ with variance $\Sigma'$ where for $1 \leq j, k \leq d$
\begin{equation}\label{eq: exp variance}
    \Sigma'_{jk}=\E (h_j h_k(L_1)), 
\end{equation}
then there exists a d-dimensional standard Brownian motion $(B_t) = (B_t^j)_{1 \leq j \leq d}$ independent of $(L_t)$ such that the processes $(\overline{L}^n, \Gamma^n, \Gamma'^n)$ converge in law to $(L, \Gamma, (\Sigma')^{1/2} B)$, where
\[\Gamma_t = \int_0^t  F(X_s) (\Sigma')^{1/2} dB_s.\]
This result implies the stable convergence stated in Theorem \ref{Th : TCL}.

To study the convergence in law of $(\overline{L}_1^n, \Gamma'^{n}_1)$, we denote by $\Phi_n$ the characteristic function of $(\overline{L}_1^n, \Gamma'^{n}_1)$ and by $\phi_n$ the characteristic function of the $(L_{\frac{1}{n}}, \frac{1}{\sqrt{n}} H(n^{1/ \alpha_0} L_{\frac{1}{n}}))$. \\
Then we have
\[ \log \Phi_n = n \log \phi_n, \]
and we just have to study the asymptotic behaviour of $\phi_n$. By definition
\[ \forall u \in \R, \; v \in \R^d, \quad \phi_n(u,v)= \E \left(e^{iuL_{\frac{1}{n}} + i \sum_{j=1}^d v_j \frac{1}{\sqrt{n}} h_j(n^{1/ \alpha_0}L_{\frac{1}{n}})} \right).\]
A Taylor expansion of the exponential function gives 
\begin{align*}
    \phi_n(u, v) & = 
    \E e^{iuL_{\frac{1}{n}}} 
    + i \frac{1}{\sqrt{n}} \sum_{j=1}^d v_j \E e^{iuL_{\frac{1}{n}}} h_j(n^{1/ \alpha_0}L_{\frac{1}{n}})
    - \frac{1}{2n} \sum_{j=1}^d v_j^2 \E e^{iuL_{\frac{1}{n}}}h_j^2 (n^{1/ \alpha_0} L_{\frac{1}{n}}) \\
    & - \frac{1}{n} \sum_{1 \leq j < k \leq d} v_j v_k \E e^{iuL_{\frac{1}{n}}}h_j h_k (n^{1/ \alpha_0} L_{\frac{1}{n}}) + o(1/n),
\end{align*}
where we used that $ \forall \; 1 \leq j \leq d, \quad \forall \; p > 0, \quad \E\left( |h_j(L_1)|^p \right) < +\infty$ to get
\[ \forall \; 1 \leq j, k, l \leq d  \quad \frac{1}{n \sqrt{n}} \left|\E ( e^{iuL_{\frac{1}{n}}} h_j h_k h_l (n^{1/ \alpha_0} L_{\frac{1}{n}}))\right| \leq \frac{1}{n \sqrt{n}} \E ( \left| h_j h_k h_l \right| (L_{1})) = o(1/n).  \]
We now study each term in the expansion of $\phi_n$.

\noindent
1. $\E e^{iuL_{\frac{1}{n}}} = (\E e^{iuL_{1}})^{1/n} = 1 + \psi(u)/n + o(1/n)$ where $\psi$ is the L\'evy-Khintchine exponent of $L_1$.

\noindent
2. For $1 \leq j \leq d$ using the self-similarity property $\E e^{iu L_{\frac{1}{n}}} h_j (n^{1/ \alpha_0} L_{\frac{1}{n}}) = \E e^{iu L_1 / n^{1/ \alpha_0}} h_j (L_1)$. But 
    $\left| \E (e^{iu L_1 / n^{1/ \alpha_0}} -1) h_j (L_1) \right| \leq \frac{|u|}{n^{1/ \alpha_0}} \E( \left|L_1 h_j (L_1)\right|)$, and $\E( \left|L_1 h_j (L_1)\right|) < +\infty$ by H\"{o}lder's inequality using $\alpha_0 > 1$. Hence $\E e^{iu L_{\frac{1}{n}}} h_j (n^{1/ \alpha_0} L_{\frac{1}{n}}) = o(1/\sqrt{n})$ since $ \E h_j (L_1) =0.$

\noindent
3. For $1 \leq j \leq d$ using the self-similarity property and by dominated convergence
   $$
   \E e^{iu L_{\frac{1}{n}}}h_j^2 (n^{1/ \alpha_0} L_{\frac{1}{n}}) = \E e^{iu L_1 / n^{1/ \alpha_0}} h_j^2 (L_1) = \E h_j^2 (L_1) + o(1).
   $$

\noindent
4. In the same way for $1 \leq j,k \leq d$, $\quad \E e^{iuL_{\frac{1}{n}}}h_j h_k (n^{1/ \alpha_0} L_{\frac{1}{n}}) = \E h_j h_k (L_1) + o(1).$
    
Putting all these results together, we finally obtain the convergence
\begin{align*}
    \log \Phi_n = n \log \phi_n & \to \psi(u) - \frac{1}{2} \sum_{j=1}^d v_j^2 \E h_j^2 (L_{1}) - \sum_{1 \leq j < k \leq d} v_j v_k \E h_j h_k (L_{1}),
\end{align*}
and we get the convergence in law of the vector $(\overline{L}_1^n, \Gamma'^{n}_1)$ to $(L_1, \gamma')$ where $\gamma'$ is a Gaussian variable independent of $L_1$ with variance
where $\gamma'$ is a Gaussian variable independent of $L_1$ with variance $\gamma'$ defined by \eqref{eq: exp variance}. This achieves the proof of Theorem \eqref{Th : TCL}.

\subsection{Proof of Theorem \ref{Th-local}} \label{Ss:Thlocal}

To prove  existence, consistency and asymptotic normality of estimating functions based estimators, we  adapt to our framework the sufficient conditions established in S{\o}rensen \cite{Sorensen}  (to obtain Theorem 2.3, Corollary 2.5 and Theorem 2.8), we also refer to Jacod and S{\o}rensen \cite{JacodSorensenUniforme}. We recall that the estimating function $G_n$ is given by \eqref{E:Gn1}-\eqref{E:Gn4},
 that $u_n$ is defined in \eqref{eq:un} and  we set $J_n(\theta)=\nabla_{\theta} G_n(\theta)$.
With this notation, Theorem \ref{Th-local} is a consequence of the two following sufficient conditions :

\begin{enumerate}
    \item $\forall \eta > 0$ we have the following convergence in probability
\[\sup_{\theta \in \{|| u_n^{-1}(\theta - \theta_0) ||\leq \eta\} } \left| \left|  u_n^T J_n(\theta) u_n - I_{\overline{v}}(\theta_0) \right|\right| \to 0,\]
where $ I_{\overline{v}}(\theta_0)$ is a positive definite matrix.

\item $(u_n^T G_n(\theta_0))_n$ stably converges in law with respect to the $\sigma$-field $\sigma(L_s, s \leq 1)$ to $I_{\overline{v}}(\theta_0)^{1/2} \mathcal{N}$ where $\mathcal{N}$ is a standard Gaussian variable independent of $I_{\overline{v}}(\theta_0)$.
\end{enumerate}
The matrix  $J_n(\theta)= \nabla_{\theta} G_n(\theta)$ has the following expression
\begin{equation} \label{E:defJn}
J_n (\theta) =  \begin{pmatrix}
\nabla_\theta G_n^1 (\theta)^T \\
\nabla_\theta G_n^2 (\theta)^T \\
\nabla_\theta G_n^3 (\theta)^T \\
\nabla_\theta G_n^4 (\theta)^T
\end{pmatrix}  = \begin{pmatrix}
J_n^{11} (\theta) &  J_n^{21} (\theta)^T \\
J_n^{21} (\theta) &  J_n^{22} (\theta)
\end{pmatrix},
\end{equation}
with
\[ J_n^{11} (\theta) = \frac{n^{2/ \alpha}}{n^2} \sum_{i=1}^n \begin{pmatrix}
  -  \frac{1}{\delta^2 X_{\frac{i-1}{n}}^{2/ \alpha}} h'_\alpha (z^n_i(\theta))
    & \text{symm} \\
    \frac{X_{\frac{i-1}{n}}}{\delta^2 X_{\frac{i-1}{n}}^{2/ \alpha}} h'_\alpha (z^n_i(\theta))
    & - \frac{X_{\frac{i-1}{n}}^2}{\delta^2 X_{\frac{i-1}{n}}^{2/ \alpha}} h'_\alpha (z^n_i(\theta))
\end{pmatrix},\]
\[ J_n^{21} (\theta) = \frac{ n^{1/ \alpha}}{n} \sum_{i=1}^n \begin{pmatrix}
   - \frac{1}{\delta^2 X_{\frac{i-1}{n}}^{1/ \alpha}} k'_\alpha (z^n_i(\theta))
    & \frac{X_{\frac{i-1}{n}}}{\delta^2 X_{\frac{i-1}{n}}^{1/ \alpha}} k'_\alpha (z^n_i(\theta)) \\
   -  \frac{ \ln \left( n / X_{\frac{i-1}{n}} \right)}{\delta \alpha^2 X_{\frac{i-1}{n}}^{1/ \alpha}} k'_\alpha (z^n_i(\theta)) + \frac{1}{\delta X_{\frac{i-1}{n}}^{1/ \alpha}}  f'_\alpha (z^n_i(\theta))
    & J_{n,i,22}^{21} (\theta)
     \end{pmatrix},\]
\[ J_n^{22} (\theta) =  \sum_{i=1}^n \begin{pmatrix}
 -   \frac{ 1}{ \delta^2} [ k_\alpha (z^n_i(\theta)) +z^n_i(\theta) k'_\alpha (z^n_i(\theta))]
    & \text{symm} \\
  -  \frac{\ln \left( n / X_{\frac{i-1}{n}} \right)}{\alpha^2 \delta} z^n_i(\theta) k'_\alpha (z^n_i(\theta)) + \frac{z^n_i(\theta)}{\delta} f'_\alpha (z^n_i(\theta))
    & J_{n,i,22}^{22} (\theta))
\end{pmatrix},\]
\begin{align*}
 J_{n,i,22}^{21} (\theta)    = \frac{\ln \left( n / X_{\frac{i-1}{n}} \right)X_{\frac{i-1}{n}}}{\delta \alpha^2 X_{\frac{i-1}{n}}^{1/ \alpha}} k'_\alpha (z^n_i(\theta)) - \frac{X_{\frac{i-1}{n}}}{\delta X_{\frac{i-1}{n}}^{1/ \alpha}} f'_\alpha (z^n_i(\theta)), 
 \end{align*}
 \begin{align*}
    J_{n,i,22}^{22} (\theta) = - \partial_\alpha f_\alpha (z^n_i(\theta)) - \frac{ \left( \ln ( n / X_{\frac{i-1}{n}}) \right)^2}{\alpha^4}  \left(z^n_i(\theta) k'_\alpha (z^n_i(\theta)) \right) \\
    + \frac{ \ln \left( n / X_{\frac{i-1}{n}} \right)}{\alpha^2}  \left( -\frac{2}{\alpha} k_\alpha (z^n_i(\theta)) + 2z^n_i(\theta) f'_\alpha (z^n_i(\theta)) \right).
 \end{align*}

\noindent
1. \underline{Convergence of $u_n^T J_n(\theta) u_n$.} We will prove the convergence in probability
\[\sup_{(a,b) \in A, \; (\delta, \alpha) \in W_n^{(\eta)} } \left| \left|  u_n^T J_n(\theta) u_n - I_{\overline{v}}(\theta_0) \right|\right| \xrightarrow{} 0,\]
for $A$ a compact subset of $(0, +\infty) \times \R$ and $W_n^{(\eta)}$ defined in \eqref{Voisinage Wn}.
 Using the expressions of $J_n$, $u_n$ and $r_n$ (defined in \eqref{E:def-vr}), we obtain after some calculus 
\begin{equation} \label{un Jn un}
    u_n^T J_n(\theta) u_n = \begin{pmatrix}
\frac{n^{2/ \alpha}}{n^{2/ \alpha_0}} I_n^{11}(\theta) &    \frac{n^{1/ \alpha}}{n^{1/ \alpha_0}} I_n^{21}(\theta)^T r_n(\theta) v_n   \\
\frac{n^{1/ \alpha}}{n^{1/ \alpha_0}} v_n^T r_n^T(\theta) I_n^{21}(\theta)
& v_n^T r_n^T(\theta) I_n^{22}(\theta) r_n(\theta) v_n + v_n^T R_n(\theta) v_n
\end{pmatrix} ,
\end{equation}
with
\[I_n^{11}(\theta) =  \frac{1}{n} \sum_{i=1}^n \begin{pmatrix}
- \frac{ 1}{\delta^2 X_{\frac{i-1}{n}}^{2/ \alpha}} h'_{\alpha} (z^n_i(\theta)) &  \text{symm} \\
\frac{X_{\frac{i-1}{n}}}{\delta^2 X_{\frac{i-1}{n}}^{2/ \alpha}} h'_{\alpha}(z^n_i(\theta)) & - \frac{ X_{\frac{i-1}{n}}^2}{\delta^2 X_{\frac{i-1}{n}}^{2/ \alpha}} h'_{\alpha} (z^n_i(\theta))
\end{pmatrix}, \]

\[ I_n^{21}(\theta) = \frac{1}{n} \sum_{i=1}^n  \begin{pmatrix}
  -  \frac{ 1}{\delta X_{\frac{i-1}{n}}^{1/ \alpha}} k'_\alpha (z^n_i(\theta))
    &  \frac{X_{\frac{i-1}{n}}}{\delta X_{\frac{i-1}{n}}^{1/ \alpha}} k'_\alpha (z^n_i(\theta)) \\
\frac{\ln \left(X_{\frac{i-1}{n}} \right)}{\delta \alpha^2 X_{\frac{i-1}{n}}^{1/ \alpha}}  k'_\alpha (z^n_i(\theta)) + \frac{1}{\delta X_{\frac{i-1}{n}}^{1/ \alpha}} f'_\alpha (z^n_i(\theta))
    &  I_{n,i,22}^{21}(\theta) 
    \end{pmatrix},  \]

\[I_n^{22}(\theta)=\frac{1}{n} \sum_{i=1}^n  \begin{pmatrix} 
 - z^n_i(\theta) k'_{\alpha} (z^n_i(\theta)) & \text{symm} \\
  \frac{ \ln(X_{\frac{i-1}{n}})}{\alpha^2}
z^n_i(\theta) k'_{\alpha} (z^n_i(\theta)) + z^n_i(\theta) f'_{\alpha} (z^n_i(\theta)) & I_{n,i,22}^{22}(\theta)
\end{pmatrix}, \]
\[
I_{n,i,22}^{21}(\theta) = - \frac{  \ln \left(X_{\frac{i-1}{n}}\right) X_{\frac{i-1}{n}} }{\delta \alpha^2 X_{\frac{i-1}{n}}^{1/ \alpha}}  k'_\alpha (z^n_i(\theta)) - \frac{X_{\frac{i-1}{n}}}{\delta X_{\frac{i-1}{n}}^{1/ \alpha}}  f'_\alpha (z^n_i(\theta)),
\]
\[I_{n,i,22}^{22}(\theta) = - \frac{ (\ln(X_{\frac{i-1}{n}}))^2}{\alpha^4}z^n_i(\theta) k'_{\alpha} (z^n_i(\theta)) 
- 2 \frac{\ln(X_{\frac{i-1}{n}})}{\alpha^2}z^n_i(\theta) f'_{\alpha} (z^n_i(\theta)) 
-  \partial_{\alpha} f_{\alpha} (z^n_i(\theta)), \]

\[R_n(\theta) = \frac{1}{n} \sum_{i=1}^n \begin{pmatrix}
    
- \frac{1 }{\delta^2} k_{\alpha} (z^n_i(\theta)) & 0 \\
0 &
\frac{2 }{\alpha^3} k_{\alpha} (z^n_i(\theta)) (\ln(X_{\frac{i-1}{n}}) - \ln(n))
\end{pmatrix} .\]
Now using the definition of $r_n$ and $W_n^{(\eta)}$ we have
\[ 
\sup_{(\delta,\alpha) \in W_n^{(\eta)} } \left| \left| r_n(\theta) - r_n(\theta_0) \right|\right| \leq C \ln(n)w_n, \]
hence by definition of $\overline{v}$
\[\sup_{(\delta,\alpha) \in W_n^{(\eta)} } \left| \left| v_n r_n(\theta) - \overline{v} \right|\right| = \sup_{(\delta,\alpha) \in W_n^{(\eta)} } \left| \left| v_n r_n(\theta_0) - \overline{v} + v_n(r_n(\theta) - r_n(\theta_0)) \right|\right| \to 0.\]
Moreover from Theorem \ref{Th : LFGN}, observing that $\E (k_{\alpha_0}(L_1)) = 0$ from \eqref{esp func = 0} and that $|| v_n || \leq C \ln(n)$, it is immediate that
\begin{equation*}
    \sup_{ (a,b) \in A, \; (\delta, \alpha) \in W_n^{(\eta)} } \left| \left|  v_n^T R_n(\theta) v_n \right|\right| \to 0.
\end{equation*}
Consequently from the factorisation \eqref{un Jn un} and observing also that 
$$
\sup_{\alpha \in W_n^{(\eta)} } \left| \frac{n^{1/\alpha}}{n^{1/\alpha_0}} - 1 \right| \to 0,
$$
the uniform convergence of $u_n^T J_n(\theta) u_n$  reduces to the uniform convergence of $I_n(\theta)$,  for $\theta \in (a,b) \times W_n^{(\eta)}$. 
From Theorem \ref{Th : LFGN} combined with the connections \eqref{esp func} we deduce the convergence in probability, 
$$
 \sup_{(a,b) \in A, \; (\delta, \alpha) \in W_n^{(\eta)} } \left| \left|  I_n(\theta) - I(\theta_0) \right|\right| \to 0,
$$
with $I(\theta_0)$ defined by \eqref{E:Info}
and finally from \eqref{un Jn un}, we obtain
\[\sup_{(a,b) \in A, \; (\delta, \alpha) \in W_n^{(\eta)} } \left| \left|  u_n^T J_n(\theta) u_n - I_{\overline{v}}(\theta_0) \right|\right| \to 0.\]
It remains to check that $I(\theta_0)$ is  positive definite. 
Let $X = (x,y,z,w)^T \in \R^4$ and consider the quadratic form $Q(X)=X^T I(\theta_0) X$. An explicit calculus leads to the factorisation
\begin{align*}
   Q(X)
     = \int_0^1 \int_{\R} ds \varphi_{\alpha}(u) du \left( \frac{1}{\delta_0} \frac{1}{X_s^{1/ \alpha_0}} h_{\alpha_0}(u) x - \frac{1}{\delta_0} \frac{X_s}{X_s^{1/ \alpha_0}} h_{\alpha_0}(u) y  \right.\\
   \left.  + k_{\alpha_0}(u) z - [f_{\alpha_0}(u) + \frac{\ln(X_s)}{\alpha_0} k_{\alpha_0}(u)]w \right)^2 ,
\end{align*}
and the result follows.
\begin{rem} \label{R:Jn}
If $v_n=r_n(\theta_0)^{-1}$, then $\overline{v}=Id_2$ and 
\[\sup_{(\delta,\alpha) \in W_n^{(\eta)} } \left| \left| v_n r_n(\theta) - \overline{v} \right|\right| \leq (\ln n)^2 w_n.\]
So in that case we obtain the convergence in probability   
$$
\forall q >0, \quad \sup_{(a,b) \in A, \; (\delta, \alpha) \in W_n^{(\eta)} } (\ln n)^q\left| \left|  u_n^T J_n(\theta) u_n - I(\theta_0) \right|\right| \xrightarrow{} 0.
$$
\end{rem}

\noindent
2. \underline{Stable convergence in law.} To study the convergence in law of $u_n^T G_n(\theta_0)$ we write
\[u_n^T G_n(\theta_0) = \frac{1}{\sqrt{n}} \begin{pmatrix}
    Id_2 & 0 \\
    0 & v_n^T r_n(\theta_0)^T
\end{pmatrix} 
\sum_{i=1}^n \begin{pmatrix}
\frac{1}{\delta_0 X_{\frac{i-1}{n}}^{1/ \alpha_0}} h_{\alpha_0}(z^n_i(\theta_0)) \\
- \frac{X_{\frac{i-1}{n}}}{\delta_0 X_{\frac{i-1}{n}}^{1/ \alpha_0}} h_{\alpha_0} (z^n_i(\theta_0)) \\
  k_{\alpha_0}(z^n_i(\theta_0)) \\
- \frac{\ln(X_{\frac{i-1}{n}})}{\alpha_0^2} k_{\alpha_0}(z^n_i(\theta_0)) - f_{\alpha_0}(z^n_i(\theta_0)) 
\end{pmatrix} ,\]
where $h_{\alpha_0}, k_{\alpha_0}$ and $f_{\alpha_0}$ satisfy Assumption \ref{A:conditions h} and from \eqref{esp func = 0} 
\[\E (h_{\alpha_0}(L_1)) = \E (k_{\alpha_0}(L_1)) = \E (f_{\alpha_0}(L_1)) = 0 .\] 
So we can apply
 Theorem \ref{Th : TCL} with
\[F(x)=\begin{pmatrix} \frac{1}{\delta_0 x^{1/ \alpha_0}} & 0 & 0 & 0 \\
0 & \frac{-x}{\delta_0 x^{1/ \alpha_0}} & 0 & 0 \\
0 & 0 & 1 & 0 \\
0 & 0 & - \frac{\ln(x)}{\alpha_0^2} & -1
\end{pmatrix}, \quad H = \begin{pmatrix} h_{\alpha_0} \\ h_{\alpha_0} \\
    k_{\alpha_0} \\ f_{\alpha_0}
\end{pmatrix},\]
and since by assumption  $ r_n(\theta_0)v_n \rightarrow  \overline{v}$, we conclude that $u_n^T G_n(\theta_0)$ stably converges in law with respect to the $\sigma$-field $\sigma(L_s, s \leq 1)$ to $I_{\overline{v}}(\theta_0)^{1/2} \mathcal{N}$ where $\mathcal{N}$ is a standard Gaussian variable independent of $I_{\overline{v}}(\theta_0)$.

\subsection{Proof of Theorem \ref{Th-global}}
We consider the normalised criteria
\begin{align} \label{E:normGn}
    \overline{G}_n(a,b, \delta, \alpha) & = \frac{n}{n^{2/ \alpha}}\begin{pmatrix}
G_n^1 (\theta) \\
G_n^2 (\theta)
\end{pmatrix}
= \frac{1}{n^{1/ \alpha}} \sum_{i=1}^n \begin{pmatrix}
     \frac{1}{\delta X_{\frac{i-1}{n}}^{1/\alpha}} h_{\alpha}(z_i^n(\theta)) \\
     \frac{- X_{\frac{i-1}{n}}}{\delta X_{\frac{i-1}{n}}^{1/ \alpha}} h_{\alpha}(z_i^n(\theta))
\end{pmatrix},
\end{align}
and we set
\begin{equation*} 
\overline{G}^{(d)}_n(a,b)=  \overline{G}_n(a,b, \tilde{\delta}_n, \tilde{\alpha}_n). 
\end{equation*} 
Solving $G^{(d)}_n(a,b)=0$ is equivalent to solve $\overline{G}^{(d)}_n(a,b) = 0$ and using Theorem 2.7.a) of Jacod and S{\o}rensen \cite{JacodSorensenUniforme}, the global uniqueness result is a consequence of the following  conditions (recalling that $(a_0,b_0)$ is an interior point of $A$, where $A$ is a compact subset of $(0, +\infty) \times \R$) :
\begin{enumerate}  
\item[(i)]  there exists $\overline{G}^{(d)}$ defined on $A$, continuously differentiable, such that we have the  convergence in probability $\overline{G}^{(d)}_n(a_0, b_0) \to 0$, $\overline{G}^{(d)}(a_0, b_0)=0$ and $(a_0, b_0)$ is the unique root of $\overline{G}^{(d)}(a,b)=0$,
    \item[(ii)]  we have the convergence in probability
    \[\sup_{(a,b) \in A} \left|\left| \nabla_{(a,b)} \overline{G}^{(d)}_n(a,b) - \nabla_{(a,b)} \overline{G}^{(d)}(a,b) \right|\right| \to 0, \]
   and $\nabla_{(a,b)} \overline{G}^{(d)}(a_0,b_0)$ is non-singular with probability one.
\end{enumerate}

We start by proving (ii). A simple computation gives
\[ \nabla_{(a,b)} \overline{G}^{(d)}_n(a,b) = \frac{1}{n} \sum_{i=1}^n \begin{pmatrix}
     \frac{-1}{\tilde{\delta}_n^2 X_{\frac{i-1}{n}}^{2/ \tilde{\alpha}_n}} h'_{\tilde{\alpha}_n}(z_i^n(a,b, \tilde{\delta}_n, \tilde{\alpha}_n)) 
     & \text{symm} \\
     \frac{X_{\frac{i-1}{n}}}{\tilde{\delta}_n^2 X_{\frac{i-1}{n}}^{2/ \tilde{\alpha}_n}} h'_{\tilde{\alpha}_n}(z_i^n(a,b, \tilde{\delta}_n, \tilde{\alpha}_n)) 
     & \frac{- X_{\frac{i-1}{n}}^2}{\tilde{\delta}_n^2 X_{\frac{i-1}{n}}^{2/ \tilde{\alpha}_n}} h'_{\tilde{\alpha}_n}(z_i^n(a,b, \tilde{\delta}_n, \tilde{\alpha}_n)) 
\end{pmatrix},\]
and we set
\[ \overline{G}^{(d)}(a,b) = \begin{pmatrix}
      \left( (a_0 - a) \int_0^1 \frac{ds}{\delta_0^2 X_s^{2/ \alpha_0}} - (b_0 - b) \int_0^1 \frac{X_s ds}{\delta_0^2 X_s^{2/ \alpha_0}} \right) \E(h'_{\alpha_0}(L_1)) \\
     - \left( (a_0 - a) \int_0^1 \frac{X_s ds}{\delta_0^2 X_s^{2/ \alpha_0}} - (b_0 - b) \int_0^1 \frac{X_s^2 ds}{\delta_0^2 X_s^{2/ \alpha_0}} \right) \E(h'_{\alpha_0}(L_1))
\end{pmatrix}.\]
For  $\eta > 0$, we consider the neighborhood of $(\delta_0, \alpha_0)$
$$
W_n^{(\eta)} = \{ (\delta, \alpha) ; \left| \left| \frac{\sqrt{n}}{\ln(n)} \begin{pmatrix} \delta - \delta_0 \\
 \alpha - \alpha_0 \end{pmatrix} \right| \right| \leq \eta\}.
$$
$W_n^{(\eta)}$ satisfies \eqref{Voisinage Wn} and since $\frac{\sqrt{n}}{\ln n}(\tilde{\delta}_n-\delta_0, \tilde{\alpha}_n- \alpha_0)$ is tight, we have
$$
\sup_n \PP( (\tilde{\delta}_n, \tilde{\alpha}_n) \notin  W_n^{(\eta)}) \xrightarrow[ \eta \rightarrow \infty]{} 0.
$$
So introducing the localisation ${\bf 1}_{(\tilde{\delta}_n, \tilde{\alpha}_n) \in  W_n^{(\eta)}}$, we just have to prove the convergence in probability
\[\sup_{(a,b) \in A, \; (\delta, \alpha) \in W_n^{(\eta)}} \left|\left| \nabla_{(a,b)} \overline{G}_n(a,b, \delta, \alpha) - \nabla_{(a,b)} \overline{G}^{(d)}(a,b) \right|\right| \to 0.\]
This convergence is immediate from Theorem \ref{Th : LFGN}.
Moreover, from \eqref{esp func} we check 
\[ \nabla_{(a,b)} \overline{G}^{(d)}(a_0,b_0) = I^{11} (\theta_0),\]
 which is non-singular from Theorem \ref{Th-local}.

We now turn to (i). We have from Taylor's formula 
\begin{align} \label{E:taylor}
    \overline{G}^{(d)}_n(a_0, b_0) &  = \overline{G}_n(a_0, b_0, \delta_0, \alpha_0) \\
     & + \int_0^1  \nabla_{(\delta, \alpha)} \overline{G}_n(a_0, b_0, \delta_0 + t (\tilde{\delta}_n  - \delta_0), \alpha_0 + t(\tilde{\alpha}_n - \alpha_0)) dt \begin{pmatrix}
\tilde{\delta}_n  - \delta_0 \\
\tilde{\alpha}_n - \alpha_0
\end{pmatrix}. \nonumber
\end{align}
Since $\E (h_{\alpha_0} (L_1)) = 0$, we deduce from Theorem \ref{Th : TCL} that $\frac{n^{1/ \alpha_0}}{\sqrt{n}}\overline{G}_n(a_0, b_0, \delta_0, \alpha_0)$ converges in law and  using  $1/ \alpha_0 > 1/2$, we conclude  that $\overline{G}_n(a_0, b_0, \delta_0, \alpha_0)$ converges in probability to zero. 

For the second term, introducing as previously the localisation ${\bf 1}_{(\tilde{\delta}_n, \tilde{\alpha}_n) \in  W_n^{(\eta)}}$, we just have to prove the convergence in probability
\begin{equation} \label{E:term2}
\sup_{(\delta, \alpha) \in W_n^{(\eta)}}  \left|\left| \int_0^1  \nabla_{(\delta, \alpha)} \overline{G}_n(a_0, b_0, \delta_0 + t (\delta  - \delta_0), \alpha_0 + t(\alpha - \alpha_0)) dt \begin{pmatrix}
\delta  - \delta_0 \\
\alpha - \alpha_0
\end{pmatrix} \right|\right| \rightarrow 0.
\end{equation}
But from \eqref{E:normGn}
\[  \nabla_{(\delta, \alpha)} \overline{G}_n (a,b,\delta, \alpha) = \frac{1}{n^{1/ \alpha}} \sum_{i=1}^n \begin{pmatrix}
    \frac{- 1}{\delta^2 X_{\frac{i-1}{n}}^{1/ \alpha}} k'_\alpha (z_i^n(\theta))
    & \partial_{\alpha} \overline{G}^1_{n,i} (\theta)\\
     \frac{X_{\frac{i-1}{n}}}{\delta^2 X_{\frac{i-1}{n}}^{1/ \alpha}} k'_\alpha (z_i^n(\theta))
    & \partial_{\alpha} \overline{G}^2_{n,i} (\theta)
\end{pmatrix},\]
with
$$
\partial_{\alpha} \overline{G}^1_{n,i} (\theta)=\frac{- \ln \left( n / X_{\frac{i-1}{n}} \right)}{\delta \alpha^2 X_{\frac{i-1}{n}}^{1/ \alpha}} k'_\alpha (z_i^n(\theta)) 
    + \frac{1}{\delta X_{\frac{i-1}{n}}^{1/ \alpha}}  f'_\alpha (z_i^n(\theta)) 
    + \frac{2 \ln(n)}{\delta \alpha^2 X_{\frac{i-1}{n}}^{1/ \alpha}} h_\alpha (z_i^n(\theta)), 
$$
$$
\partial_{\alpha} \overline{G}^2_{n,i} (\theta)=\frac{\ln \left( n / X_{\frac{i-1}{n}} \right)X_{\frac{i-1}{n}}}{\delta \alpha^2 X_{\frac{i-1}{n}}^{1/ \alpha}} k'_\alpha (z_i^n(\theta)) 
    - \frac{X_{\frac{i-1}{n}}}{\delta X_{\frac{i-1}{n}}^{1/ \alpha}} f'_\alpha (z_i^n(\theta))
    - \frac{2 \ln(n) X_{\frac{i-1}{n}}}{\delta \alpha^2 X_{\frac{i-1}{n}}^{1/ \alpha}} h_\alpha (z_i^n(\theta)).
$$
 Using Theorem \ref{Th : LFGN}, we see from the  expression of $\nabla_{(\delta, \alpha)} \overline{G}_n$ that
 $$
 \sup_{(\delta, \alpha) \in W_n^{(\eta)}} \left|\left| \frac{n^{1/ \alpha}}{n \ln(n)} \nabla_{(\delta, \alpha)} \overline{G}_n (a_0, b_0,\delta, \alpha) \right|\right|, $$
  is tight and since
\[ \sup_{(\delta, \alpha) \in W_n^{(\eta)}}  \left|\left| \frac{n \ln(n)}{n^{1/\alpha}} \begin{pmatrix} 
\delta  - \delta_0 \\
\alpha - \alpha_0
\end{pmatrix} \right|\right| \leq C  \frac{(\ln n)^{2}}{n^{1/ \alpha_0 - 1/2}} \to 0,\]
we deduce \eqref{E:term2} and we  conclude that $\overline{G}^{(d)}_n(a_0, b_0)$ converges to zero in probability. Finally, from the expression of $\overline{G}^{(d)}$, we see immediately that $(a_0, b_0)$ is the unique root of $\overline{G}^{(d)}(a,b)=0$ and (i) is proved.

To finish the proof of Theorem \ref{Th-global}, it remains to show that $ \frac{n^{1/ \alpha_0 }}{\sqrt{n}(\ln n)^{2}} ( \hat{a}_n - a_0, \hat{b}_n - b_0)$
    is tight. We have $ \overline{G}^{(d)}_n(\hat{a}_n, \hat{b}_n) = 0$ on a set $D_n$ with $\PP(D_n) \to 1$, and from Taylor's formula, we obtain on $D_n$
  \begin{align*}
    - \overline{G}^{(d)}_n(a_0, b_0) = 
     \int_0^1  \nabla_{(a, b)} \overline{G}^{(d)}_n(a_0 + t (\hat{a}_n - a_0), b_0 + t (\hat{b}_n - b_0)) dt \begin{pmatrix}
\hat{a}_n  - a_0 \\
\hat{b}_n - b_0
\end{pmatrix}.
\end{align*}  
 From ii) and using that $(\hat{a}_n, \hat{b}_n)$ converges in probability to $(a_0, b_0)$, we have the convergence in probability
\[ \int_0^1  \nabla_{(a, b)} \overline{G}^{(d)}_n(a_0 + t (\hat{a}_n - a_0), b_0 + t (\hat{b}_n - b_0)) dt \rightarrow \nabla_{(a,b)} \overline{G}^{(d)}(a_0,b_0) = I^{11} (\theta_0),\]   
where $I^{11} (\theta_0)$ is non-singular.  

Consequently, we just have to prove the tightness of   $\frac{n^{1/ \alpha_0 }}{\sqrt{n}(\ln n)^{2}}\overline{G}^{(d)}_n(a_0, b_0)$. We study each term of  the expansion \eqref{E:taylor}. As in the proof of (i), we see from Theorem \ref{Th : TCL} that $\frac{n^{1/ \alpha_0}}{\sqrt{n} (\ln n)^2}\overline{G}_n(a_0, b_0, \delta_0, \alpha_0)$  converges in probability to zero. Turning to the second term in \eqref{E:taylor} and proceeding as in (i), we have the tightness of $\sup_{(\delta, \alpha) \in W_n^{(\eta)}} \left|\left| \frac{n^{1/ \alpha}}{n \ln(n)} \nabla_{(\delta, \alpha)} \overline{G}_n (a_0, b_0,\delta, \alpha) \right|\right| $ and by definition of $W_n^{(\eta)}$
\[ \sup_{ (\delta, \alpha) \in W_n^{(\eta)}} \left|\left|  \frac{n \ln(n)}{n^{1/\alpha}} \begin{pmatrix} 
\tilde{\delta}_n  - \delta_0 \\
\tilde{\alpha}_n - \alpha_0
\end{pmatrix}\right|\right| \leq C \frac{\sqrt{n}(\ln n)^{2}}{n^{1/ \alpha_0}}.\]
This shows the tightness of
\[ \frac{n^{1/ \alpha_0}}{\sqrt{n} (\ln n)^2} \int_0^1  \nabla_{(\delta, \alpha)} \overline{G}_n(a_0, b_0, \delta_0 + t (\tilde{\delta}_n  - \delta_0), \alpha_0 + t(\tilde{\alpha}_n - \alpha_0)) dt \begin{pmatrix}
\tilde{\delta}_n  - \delta_0 \\
\tilde{\alpha}_n - \alpha_0
\end{pmatrix}.\]
We conclude that $\frac{n^{1/ \alpha_0}}{\sqrt{n} (\ln n)^2}(
\hat{a}_n  - a_0 ,\hat{b}_n - b_0)$
 is tight.

\subsection{Proof of Theorem \ref{Estimation delta}}
Using that
\[\tilde{\delta}_n - \delta_0 = (\tilde{\delta}_n(\frac{1}{2}) - \delta_0^{1/2}) (\tilde{\delta}_n(\frac{1}{2}) + \delta_0^{1/2}), \]
the proof reduces to show that $\tilde{\delta}_n(\frac{1}{2}) \xrightarrow{} \delta_0^{1/2}$ and that $\frac{\sqrt{n}}{\ln(n)} (\tilde{\delta}_n(\frac{1}{2}) - \delta_0^{1/2})$ is tight. We recall that $\tilde{\alpha}_n := \tilde{\alpha}_n(\frac{1}{2})$ and $\tilde{\delta}_n(\frac{1}{2})$ are respectively  defined by \eqref{Estimateur alpha} and \eqref{E:pre-delta-p}.
We start with the decomposition
\begin{align} \label{E:delta}
    \tilde{\delta}_n(\frac{1}{2}) = \frac{1}{m_{\frac{1}{2}}(\alpha_0)} 
    \left( 1 + \epsilon^1(\tilde{\alpha}_n) \right) 
    \left( 1 + \epsilon_n^2(\tilde{\alpha}_n) \right)
    \frac{1}{n} \sum_{i=2}^n A^n_i 
    \left( 1 + (X_{\frac{i-2}{n}}^{(1/ \alpha_0 - 1/ \hat{\alpha}_n)/2} - 1 ) \right),
\end{align}
where 
$$
\epsilon^1(\tilde{\alpha}_n)= \frac{m_{\frac{1}{2}}(\alpha_0) - m_{\frac{1}{2}}(\tilde{\alpha}_n)}{m_{\frac{1}{2}}(\tilde{\alpha}_n)},  \quad \quad \epsilon_n^2(\tilde{\alpha}_n)= n^{ (1/ \tilde{\alpha}_n - 1/ \alpha_0)/2 } -1,
$$
\begin{align*}
 A^n_i =  n^{1/ 2\alpha_0} \frac{\left|\Delta_i^n X- \Delta_{i-1}^n X\right|^{1/2} }{X_{\frac{i-2}{n}}^{1/ 2\alpha_0}}.
\end{align*}
We  will prove below that $\frac{1}{n }\sum_{i=2}^n A^n_i$ converges to $m_{\frac{1}{2}}(\alpha_0)\delta_0^{1/2}$ in probability and that $\sqrt{n} (\frac{1}{n }\sum_{i=2}^n A^n_i-m_{\frac{1}{2}}(\alpha_0)\delta_0^{1/2})$ converges stably in law. Now  since $\sqrt{n}(\tilde{\alpha}_n - \alpha_0)$ is tight, we will use the localisation $\{ \tilde{\alpha}_n \in \tilde{W}_n^{(\eta)} \} $ where  $\tilde{W}_n^{(\eta)} = \{\alpha \ ; \ \left| \sqrt{n}(\alpha - \alpha_0) \right| \leq \eta\} $. 
So we deduce (recalling \eqref{E:mp}) 
\[ \sup_{ \alpha \in \tilde{W}_n^{(\eta)}}\vert \epsilon^1(\alpha) \vert \leq \frac{C}{\sqrt{n}}, \quad  \sup_{ \alpha \in \tilde{W}_n^{(\eta)}} \vert \epsilon_n^2(\alpha) \vert \leq \frac{C \ln(n)}{\sqrt{n}}, \]
\[ \sup_{ \alpha \in \tilde{W}_n^{(\eta)}}\left| X_{\frac{i-2}{n}}^{(1/ \alpha_0 - 1/ \hat{\alpha}_n)/2} - 1 \right| \leq  \frac{C}{\sqrt{n}} \left|\ln(X_{\frac{i-2}{n}})\right|. \]
Combining these results with \eqref{E:delta}, we can write
\begin{align*}
    \tilde{\delta}_n(\frac{1}{2}) - \delta_0^{1/2} =   \frac{1}{m_{\frac{1}{2}}(\alpha_0)} \frac{1}{n} \sum_{i=2}^n A_i - \delta_0^{1/2} 
     + \frac{1}{m_{\frac{1}{2}}(\alpha_0)} \epsilon_n^2(\tilde{\alpha}_n) \frac{1}{n} \sum_{i=2}^n A_i
+ R_n, 
\end{align*}
where $\frac{\sqrt{n}}{\ln(n)}R_n  \to 0$ in probability. 
But from Taylor's formula 
$$
\epsilon_n^2(\tilde{\alpha}_n)=- \ln(n)(\tilde{\alpha}_n - \alpha_0) \frac{1}{2 \tilde{\alpha}_n \alpha_0} e^{c_n} \; \text{with} \; |c_n| \leq \frac{1}{2}\ln(n) |1/ \tilde{\alpha}_n - 1/ \alpha_0|,
$$  
 and recalling that $\sqrt{n}( \tilde{\alpha}_n - \alpha_0)$ stably converges in law we obtain both the convergence of $ \tilde{\delta}_n(\frac{1}{2})$ to $\delta_0^{1/2}$ and the tightness of $\frac{\sqrt{n}}{\ln(n)} ( \tilde{\delta}_n(\frac{1}{2}) - \delta_0^{1/2} )$.

{\underline{ Convergence of  $\frac{1}{n }\sum_{i=2}^n A^n_i$.} We end the proof of Theorem \ref{Estimation delta} by showing that $\frac{1}{n }\sum_{i=2}^n A^n_i$ converges to $m_{\frac{1}{2}}(\alpha_0)\delta_0^{1/2}$ in probability and that $\sqrt{n} (\frac{1}{n }\sum_{i=2}^n A^n_i-m_{\frac{1}{2}}(\alpha_0)\delta_0^{1/2})$ converges stably in law.
We have the decomposition
\begin{align*}
\frac{1}{n} \sum_{i=2}^n A_i - m_{\frac{1}{2}}(\alpha_0)\delta_0^{1/2} =  &  \frac{1}{n} \sum_{i=2}^n \left( A_i - \delta_0^{1/2} n^{1/ 2\alpha_0} |\Delta_i^n L - \Delta_{i-1}^n L|^{1/2} \right) \\
 & +  \delta_0^{1/2} \left( \frac{1}{n} \sum_{i=2}^n  n^{1/ 2\alpha_0} |\Delta_i^n L - \Delta_{i-1}^n L|^{1/2} - m_{\frac{1}{2}}(\alpha_0) \right).
\end{align*}
Following Todorov \cite{Todorov} (Term $A_1$ in the proof of Theorem 2), we have  that 
\[\sqrt{n}\left( \frac{1}{n} \sum_{i=2}^n  n^{1/ 2\alpha_0} |\Delta_i^n L - \Delta_{i-1}^n L|^{1/2} - m_{\frac{1}{2}}(\alpha_0) \right)\] converges stably in law and we conclude 
 by showing that 
 \begin{equation} \label{E:neg}
 \frac{1}{n} \sum_{i=2}^n \left( A_i - \delta_0^{1/2} n^{1/ 2\alpha_0} |\Delta_i^n L - \Delta_{i-1}^n L|^{1/2} \right) =o_P(\frac{1}{\sqrt{n}}).
 \end{equation}
  We have from \eqref{eq:EDS} after some calculus
\begin{align*}
\frac{n^{1/ \alpha_0}}{X_{\frac{i-2}{n}}^{1/ \alpha_0}} \left( \Delta_i^n X- \Delta_{i-1}^n X\right) 
 = \xi^1_i + \xi^2_i,
\end{align*}
with
\begin{align*}
\xi^1_i &= \delta_0 n^{1/ \alpha_0} (\Delta_{i}^n L - \Delta_{i-1}^n L), \\
\xi^2_i & = \frac{- b_0 n^{1/ \alpha_0}}{X_{\frac{i-2}{n}}^{1/ \alpha_0}} \left( \int_\frac{i-1}{n}^\frac{i}{n} (X_s - X_\frac{i-2}{n}) ds - \int_\frac{i-2}{n}^\frac{i-1}{n} (X_s - X_\frac{i-2}{n}) ds \right) \\
& + \frac{\delta_0 n^{1/ \alpha_0}}{X_{\frac{i-2}{n}}^{1/ \alpha_0}} \left( \int_\frac{i-1}{n}^\frac{i}{n} (X_s^{1/ \alpha_0} - X_\frac{i-1}{n}^{1/ \alpha_0}) dL_s - \int_\frac{i-2}{n}^\frac{i-1}{n} (X_s^{1/ \alpha_0} - X_\frac{i-2}{n}^{1/ \alpha_0} )dL_s \right) \\
& + \frac{\delta_0 }{X_{\frac{i-2}{n}}^{1/ \alpha_0}} (X_\frac{i-1}{n}^{1/ \alpha_0} - X_\frac{i-2}{n}^{1/ \alpha_0}) n^{1/ \alpha_0} \Delta_{i}^n L.
\end{align*}
Consequently with this notation
\begin{equation} \label{E:Ai}
\frac{1}{n} \sum_{i=2}^n \left( A_i - \delta_0^{1/2} n^{1/ 2\alpha_0} |\Delta_i^n L - \Delta_{i-1}^n L|^{1/2} \right)=\frac{1}{n} \sum_{i=2}^n( | \xi^1_i + \xi^2_i |^{1/2} - | \xi^1_i |^{1/2}).
\end{equation}
But we have the bound
\begin{align*}
 \left| |\xi^1_i + \xi^2_i|^{1/2} - |\xi^1_i|^{1/2} \right| & = \frac{| |\xi^1_i + \xi^2_i| - |\xi^1_i||}{ |\xi^1_i + \xi^2_i|^{1/2} + |\xi^1_i|^{1/2}}\\
& \leq C \left( |\xi^1_i|^{-1/2} |\xi^2_i| {\bf 1}_{|\xi^1_i| > \epsilon_n}
+ |\xi^2_i|^{1/2} {\bf 1}_{|\xi^1_i| \leq \epsilon_n} \right),
\end{align*}
where from Proposition \eqref{P:moments} and Proposition \eqref{P:intstoch}, 
\[ \forall \ 0 < p < \alpha_0, \quad \E(|\xi^2_i|^p) \leq C_p \frac{1}{n^{p/\alpha_0}}. \]
Now using the fact that the density of a symmetric $\alpha$-stable distribution is bounded, we have $\PP(|\xi^1_i| \leq \epsilon) \leq C \epsilon.$ Moreover we have $\forall \ p \in (-1, \alpha_0)$, $\E(|\xi^1_i|^p) < + \infty.$ \\
From H\"{o}lder's inequality with $p < \alpha_0$, choosing $p$ arbitrarily close to $\alpha_0$ hence $1/p'$ arbitrarily close to $\frac{\alpha_0 - 1}{\alpha_0}$ and $p'/2 > 1$
\begin{align*}
    \E \left( |\xi^1_i|^{-1/2} |\xi^2_i|{\bf 1}_{| \xi^1_i| > \epsilon_n} 
    \right) & \leq \left(\E  (|\xi^1_i|^{-p'/2} {\bf 1}_{| \xi^1_i| > \epsilon_n} )
    \right)^{1/p'} \left(\E |\xi^2_i|^p\right)^{1/p}  \\   
    & \leq \left(\E ( |\xi^1_i|^{-1 + l} |\xi^i_1|^{-p'/2 + 1 -l}{\bf 1}_{| \xi^i_1| > \epsilon_n} )
    \right)^{1/p'} \left(\E |\xi^2_i|^p\right)^{1/p} \\
    & \leq C \epsilon_n^{-1/2 + 1/p' - l/p'} \frac{1}{n^{1/ \alpha_0}} \quad \text{with } l>0 \text{ arbitrarily small}.
\end{align*}
Moreover using H\"{o}lder's inequality with $p/2 < \alpha_0$, choosing $p$ arbitrarily close to $2 \alpha_0$ hence $1/p'$ arbitrarily close to $\frac{\alpha_0-1/2}{\alpha_0}$
\begin{align*}
    \E \left( |\xi^2_i|^{1/2} {\bf 1}_{| \xi^1_i| \leq \epsilon_n} 
    \right)
    & \leq \PP(|\xi^1_i| \leq \epsilon_n)^{1/p'} \left(\E |\xi^2_i|^{p/2}\right)^{1/p} \\
    & \leq C \epsilon_n^{1/p'} \frac{1}{n^{1/2\alpha_0}}.
\end{align*}
Taking $\epsilon_n = n^{-\frac{1}{\alpha_0 + 1}}$ and observing that
$\frac{1}{\alpha_0} + \frac{1}{\alpha_0 + 1} (\frac{\alpha_0 - 1}{\alpha_0}- 1/2) > 1/2$ and $\frac{1}{ 2\alpha_0} +\frac{1}{\alpha_0 + 1} \frac{\alpha_0-1/2}{\alpha_0} > 1/2$,
we conclude that for some $l>0$
\[ \E (\left||\xi^1_i + \xi^2_i|^{1/2} - |\xi^1_i|^{1/2} \right| ) \leq \frac{C}{n^{1/2 + l}}.\]
Taking the expectation in \eqref{E:Ai} and combining with the previous inequality we obtain \eqref{E:neg}. 

\vspace{0.5cm}

{\bf This work was supported by french ANR reference ANR-21-CE40-0021.}


\end{document}